\documentclass[reqno]{amsart}
\pdfoutput=1
\usepackage{amsfonts}
\usepackage{amsthm}
\usepackage{amscd}
\usepackage{amsmath}
\usepackage{txfonts}
\usepackage{mathrsfs}
\usepackage{hyperref}

\theoremstyle{definition}
\newtheorem{defn}{Definition}[section]

\theoremstyle{plain}
\newtheorem{thm}{Theorem}[section]
\newtheorem{prop}[thm]{Proposition}
\newtheorem{lem}[thm]{Lemma}
\newtheorem{cor}[thm]{Corollary}

\newtheorem*{maintheorem}{Main Theorem}
\newtheorem*{clm}{Claim}

\theoremstyle{remark}
\newtheorem*{rmk}{Remark}

\newcommand{\abs}[1] {\left|{#1}\right|}  
\newcommand{\norm}[1] {\left\|{#1}\right\|} 
\newcommand{\brac}[1] {\left({#1}\right)}  
\newcommand{\Brac}[1] {\left\{\,{#1}\,\right\}} 

\newcommand{\pd}[2] {\dfrac{\partial {#1}}{\partial {#2}}} 

\newcommand{\mbbc}{\mathbb{C}}
\newcommand{\mbbd}{\mathbb{D}}

\newcommand{\mbbn}{\mathbb{N}}

\newcommand{\mbbr}{\mathbb{R}}

\newcommand{\mbfr}{\mathbf{r}}

\newcommand{\mclj}{\mathcal{J}}

\newcommand{\mclm}{\mathcal{M}}

\newcommand{\mclp}{\mathcal{P}}
\newcommand{\mclq}{\mathcal{Q}}
\newcommand{\mclr}{\mathcal{R}}

\newcommand{\mclt}{\mathcal{T}}

\newcommand{\mscf}{\mathscr{F}}

\newcommand{\hI}{\widehat{I}}

\newcommand{\hX}{\widehat{X}}
\newcommand{\hY}{\widehat{Y}}
\newcommand{\hZ}{\widehat{Z}}

\DeclareMathOperator{\dif}{d}
\DeclareMathOperator{\id}{id}
\DeclareMathOperator{\dom}{dom}
\DeclareMathOperator{\Leb}{Leb}

\title{Analytic skew-products of quadratic polynomials over Misiurewicz-Thurston maps}

\author{Rui Gao}
\address{Rui Gao}
\email{gaorui12@nus.edu.sg}
\author{Weixiao Shen}
\address{Weixiao Shen}
\email{matsw@nus.edu.sg}

\begin{document}
\maketitle

\begin{abstract}
We consider skew-products of quadratic maps over certain Misiurewicz-Thurston maps and study their statistical properties. We prove that, when the coupling function is a polynomial of odd degree, such a system admits two positive Lyapunov exponents almost everywhere and a unique absolutely continuous invariant probability measure.
\end{abstract}

\section{Introduction}
A quadratic polynomial $Q_c(x)=c-x^2$($1<c \le 2$) induces a unimodal map on the interval $[c-c^2, c]$.
If the critical point $0$ is strictly pre-periodic under iteration of $Q_c$, then we say that $Q_c$ is {\em Misiurewicz-Thurston}.
These maps are the simplest non-uniformly expanding dynamical systems, and their properties are well understood. See for example~\cite{N}. When considered as a holomorphic map defined on the Riemann Sphere, the complex dynamics generated by $Q_c$ is also exhaustedly studied. In this article, we shall make use of its subhyperbolicity as was studied in, for example, \cite[\S V.4]{CG} or \cite[\S 19]{M} .

In Viana \cite{V}, Misiurewicz-Thurston quadratic polynomials were used to construct non-uniformly expanding maps in dimension greater than one. He considered the following skew-product
$$G: \left(\mathbb{R}/\mathbb{Z}\right)\times \mbbr \circlearrowleft,\quad (\theta,y)\mapsto ( d \cdot\theta, Q_c(y)+\alpha\sin(2\pi \theta)),$$
where $d\ge 16$ is an integer, and $Q_c(1<c<2)$ is Misiurewicz-Thurston. He proved that $G$ has two positive Lyapunov exponents almost everywhere, provided that $\alpha>0$ is small enough. The assumption that $d\ge 16$ was weakened to $d\ge 2$ in \cite{BST}. See also~\cite{S1} for a similar result in non-integer case of $d$.

In Schnellmann ~\cite{S2}, the following skew product was considered:
\begin{equation}\label{eqn:mathscrF}
\mathscr{F}: [a-a^2,a]\times\mbbr \circlearrowleft\quad, \quad (x,y)\mapsto (g(x),Q_b(y)+\alpha \varphi(x)),
\end{equation}
where $g=Q_a^{m_1}$, $Q_a(1<a\le 2)$ and $Q_b(1<b<2)$ are Misiurewicz-Thurston, and $m_1$ is a large positive integer.
For certain coupling function $\varphi$, with singularity and depending on $a$, the author proved that $\mathscr{F}$ has two positive Lyapunov exponents almost everywhere, provided that $\alpha>0$ is small enough.

In this paper, we shall also consider systems in the form of (\ref{eqn:mathscrF}), but the coupling function $\varphi$ will be taken to be a nonconstant polynomial independent of $a$. The main result is the following:

\begin{maintheorem}\label{main}
Let $Q_a(1<a\le 2)$ and $Q_b(1<b<2)$ be Misiurewicz-Thurston and let $\varphi$ be a polynomial of odd degree. Assume also that $Q_a$ is topologically exact on $[a-a^2,a]$. Then there exists a positive integer $m_0=m_0(a)$ such that for each positive integer $m_1\ge m_0$ the following holds: For any $\alpha>0$ sufficiently small, the map $\mathscr{F}$ defined in (\ref{eqn:mathscrF})
has two positive Lyapunov exponents. Moreover, $\mathscr{F}$ has a unique invariant probability measure that is absolutely continuous with respect to the Lebesgue measure.
\end{maintheorem}

\begin{rmk}
The assumption that $\varphi$ is of odd degree seems quite artificial. Unfortunately, in our argument, this assumption cannot be fully gotten rid of. See the proof of the claim in Lemma \ref{lem:mathcalT} and the remark after that lemma for details. 
\end{rmk}

We shall use many ideas appearing in the works cited above. In particular,  we choose $m_0$ large enough, so that the map $\mathscr{F}$ is {\em partially hyperbolic}, i.e. the expansion along the $x$-direction is stronger than the $y$-direction. The exact choice of $m_0$ will be specified in (\ref{eqn:m_0}), at the end of \S\ref{subsec:subhyperbolicity}.

To better visualize the partial hyperbolicity, we shall introduce a conjugation map $u: [a-a^2,a]\to I_a$ in \S\ref{subsec:EC}, so that the conjugated map $h_0=u\circ Q_a\circ u^{-1}:I_a\circlearrowleft$ becomes uniformly expanding. Our construction of $u$ is based on subhyperbolicity of $Q_a$ near its Julia set. A similar construction, using instead the absolutely continuous invariant density, was used in \cite{S2}. We shall show that $u^{-1}$ and inverse branches of $h_0$ have very nice analytic properties, see Lemma \ref{lem:expandingcoor}. Then the map $\mathscr{F}$ defined in (\ref{eqn:mathscrF}) can be conjugate to the following map
$F$ via $u\times\id_{\mbbr}$:
\begin{equation}\label{eqn:F}
F:I_a\times\mbbr\circlearrowleft\quad,\quad (\theta,y)\mapsto (h(\theta),Q_b(y)+\alpha \phi(\theta)),
\end{equation}
where $h=h_0^{m_1}$ and $\phi=\varphi\circ u^{-1}$. We shall prove the equivalent statement of the Main Theorem for $F$ instead of $\mscf$. An important advantage of this construction is that in the $\theta$-coordinate, the derivative of an admissible curve, i.e., the image of a horizontal curve under iteration of $F$, can be approximated by an analytic function $\alpha T$, for some $T$ chosen from a certain compact space of holomorphic maps.

The basic strategy to prove our Main Theorem for the system $F$ is the same as that in previous works \cite{V,BST,S2}. The main effort is to control recurrence of typical orbits to the critical line $y=0$, see Proposition~\ref{prop:slowrec}. This proposition, together with the ``Building Expansion'' Lemmas in \cite{V}, which are summarized in Lemma \ref{lem:BE}, implies that $F$ is non-uniformly expanding in the sense of ~\cite{ABV}, and hence the results proved therein complete the proof of our Main Theorem.

Following Viana's original argument, Proposition~\ref{prop:slowrec} is reduced to showing that under iteration of $F$, a horizontal curve becomes {\em non-flat} and widely spreads in the $y$-direction. In our argument, these two properties are obtained in Proposition~\ref{prop:admissible curve} and Lemma \ref{lem:separate} respectively.
Lemma \ref{lem:mathcalT}, serving as an intermediate step,
is proved
by combining subhyperbolicity of $Q_a$ with the super attracting behavior of $Q_a$ near infinity.

It remains an interesting open problem whether the statement of our Main Theorem holds for $(a, b)$ chosen from a positive subset of $[1,2]\times [1,2)$.

\section{Preliminaries}\label{sec:preliminaries}

\subsection{Subhyperbolicity of $Q_a$}\label{subsec:subhyperbolicity}

Let us review some useful properties of the complex dynamics of the Misiurewicz-Thurston quadratic polynomial $Q_a$($1<a\le 2$), which are well discussed in many literatures, see for example \cite[\S V.3]{CG} or \cite[\S 19]{M}.

Denote the {\em post-critical} set of $Q_a$ by $PC_a$, i.e.
$$PC_a=\Brac{Q_a^n(0):n\ge 1}.$$
Denote the {\em Julia} set of $Q_a$ by $\mclj_a$. Since $Q_a$ is Misiurewicz-Thurston, its {\em Fatou} set $\overline{\mbbc}\setminus\mclj_a$ coincides with the attracting basin of infinity. Therefore, we can fix some $R=R(a)>0$, and denote $V_a=\{\,z\in\mbbc:|z|<R\,\}$, such that

\begin{equation*}\label{eqn:V_a}
   \mclj_a \subset V_a \quad\mbox{and}\quad Q_a^{-1}(\overline{V_a})\subset V_a.
\end{equation*}

An important feature of $Q_a$ as a Misiurewicz-Thurston map is the so-called  {\em subhyperbolicity}. That is to say, there exists a conformal metric $\rho_a(z)\abs{\dif z }$ around $\mclj_a$(with mild singularities at $PC_a$), such that with respect to this metric, $Q_a$ is uniformly expanding in some neighborhood of $\mclj_a$. We shall give an explicit construction of $\rho_a(z)\abs{\dif z }$ below for further usage. For this purpose, let us start with the following lemma. Because this lemma can be essentially found in \cite[\S V, Theorem 3.1]{CG} or \cite[\S 19]{M}, we merely provide an outline of its proof.

\begin{lem}\label{lem:cover}
For each Misiurewicz-Thurston $Q_a(1<a\le 2)$, there exists a branched covering  $\pi:\mbbd\to V_a$, where $\mbbd$ denotes the unit disk in $\mbbc$, satisfying the following properties:
\begin{itemize}
  \item [(i)] The collection of critical points of $\pi$ is precisely equal to $\pi^{-1}(PC_a)$, and for each $\hat{c}\in\pi^{-1}(PC_a)$, $\pi''(\hat{c})\ne 0$, i.e the \emph{local degree} of $\pi$ is constantly equal to $2$ on $\pi^{-1}(PC_a)$.
  \item [(ii)] $\pi$ is {\em normal}. That is to say, $\pi|_{\mbbd\setminus\pi^{-1}(PC_a)}$ is normal as an unbranched covering, i.e. for each $p\in V_a\setminus PC_a$, the covering transformation group of $\pi|_{\mbbd\setminus\pi^{-1}(PC_a)}$ acts transitively on $\pi^{-1}(p)$.
  \item [(iii)] For each pair of points $(\hat{p},\hat{q})$ in $\mbbd$ with $Q_a(\pi(\hat{p}))=\pi(\hat{q})$, $Q_a^{-1}$ can be lifted to a single-valued holomorphic map  $\hat{\tau}:\mbbd\circlearrowleft$, which maps $\hat{q}$ to $\hat{p}$. Moreover, $\hat{c}\in\mbbd$ is a critical point of $\hat{\tau}$ precisely if $\pi(\hat{c})\in PC_a\setminus\Brac{a}$  and $\pi(\hat{\tau}(\hat{c}))\notin PC_a$. Besides, once $\hat{\tau}'(\hat{c})= 0$, then $\hat{\tau}''(\hat{c})\ne 0$.
\end{itemize}
\end{lem}
\begin{proof}[Sketch of Proof]
Fix $q\in V_a\setminus PC_a$ as a base point. For each $v\in PC_a$, let $l_v\in\pi_1(V_a\setminus PC_a,q)$ be represented by a simple closed curve based on $q$ whose winding number around $\nu$ is one and whose  winding number around $\nu'$ is $0$ for each $\nu'\in PC_a\setminus\{v\}$. Then $\pi_1(V_a\setminus PC_a,q)$ is the free group generated by $\Brac{l_v}_{v\in PC_a}$. Let $G$ be its normal subgroup generated by $\{l_v^2\}_{v\in PC_a}$. Then there exists a normal unbranched covering $\pi:\mbbd/G\to V_a\setminus PC_a$. For any small disk $B(v,r)$, $v\in PC_a$,  each connected component of $\pi^{-1}(B(v,r)\setminus\{v\})$ is a punctured disk, and the restriction of $\pi$ to it is a double covering. By filling up the punctured points into  $\mbbd/G$, it is easy to verify that $\pi$ can be extended to a branched covering from some simply connected Riemann surface to $V_a$. Since $V_a$ is hyperbolic, we obtain a branched covering $\pi:\mbbd\to V_a$. According to the construction of $\pi$, assertions (i) and (ii) hold automatically.

Because $Q_a$ has a unique critical point $0$ with $Q_a''(0)\ne 0$, the construction of $\pi$ exactly guarantees that $Q_a^{-1}$ can be locally lifted to a holomorphic map from $\mbbd$ to $\mbbd$ with respect to $\pi$. Since $\mbbd$ is simply connected, the global existence of the lift follows from monodromy theorem. It remains to check the statement about critical points of $\hat{\tau}$. Since $Q_a\circ\pi\circ\hat{\tau}=\pi$ and since $\pi'(\hat{c})=0$ implies that $\pi''(\hat{c})\ne 0$,  $\hat{c}$ is a critical point of $\hat{\tau}$ if and only if $\hat{c}$  is a critical point of $\pi$ and $\hat{\tau}(\hat{c})$ is not a critical point of $Q_a\circ \pi$. The conclusion follows.
\end{proof}

Let $\rho_\mbbd(w)|\dif w|$ be the standard Poincar\'{e} metric on $\mbbd$, where $\rho_\mbbd(w)=\frac{2}{1-|w|^2}$. By (i) and (ii) of Lemma \ref{lem:cover}, $\rho_\mbbd(w)|\dif w|$ can be pushed forward by $\pi$ to $V_a$, which induces a metric $\rho_a(z)\abs{\dif z}$ on $V_a$, i.e.

\begin{equation}\label{eqn:rho_a}
\rho_\mbbd(w)=\rho_a(\pi(w))|\pi'(w)|,\quad\forall w\in \mbbd\setminus\pi^{-1}(PC_a).
\end{equation}
Therefore, $\rho_a$ is is a strictly positive analytic function on $V_a\setminus PC_a$, and it has a singularity of order $\abs{z-v}^{-\frac{1}{2}}$ at $v$ for each $v\in PC_a$. It follows that the function $\frac{\rho_a(Q_a(z))}{\rho_a(z)}|(Q_a)'(z)|$ is actually continuous on $Q_a^{-1}(V_a)$. Moreover, by (iii) of Lemma \ref{lem:cover} and Schwarz lemma, it is strictly larger than $1$ on $Q_a^{-1}(V_a)$. In particular, since $Q_a^{-1}(\overline{V_a})$ is a compact subset of $V_a$,

\begin{equation}\label{eqn:lambda_a}
   \lambda_a:=\inf_{z\in Q_a^{-2}(V_a)}\frac{\rho_a(Q_a(z))}{\rho_a(z)}|(Q_a)'(z)|>1.
\end{equation}

Here is the right position to specify our choice of $m_0$ in the statement of the Main Theorem. By definition, $m_0$ is the minimal positive integer such that
\begin{equation}\label{eqn:m_0}
    \lambda_a^{m_0}> 4.
\end{equation}
By considering the derivative of $Q_a$ at its fixed point $\frac{\sqrt{1+4a}-1}{2}$, it follows that

$$\lambda_a\le \sqrt{1+4a}-1\le 2.$$
For some technical reason,
we shall also need a constant $\tilde{\lambda\,}_{\!a}\in(4^{\frac{1}{m_0}},\lambda_a)$. To be definite, let
$$\tilde{\lambda\,}_{\!a}:=\frac{\lambda_a+4^{\frac{1}{m_0}}}{2}.$$

\begin{rmk}
In fact our argument is valid for a smaller $m_0$. On the one hand, we can replace $\lambda_a$ with a larger number $\mu_a>1$, which only needs to satisfy that
$$\rho_a(Q_a^n(x))|(Q_a^n)'(x)|\ge C_a\mu_a^n\cdot \rho_a(x),\quad \forall x\in[a-a^2,a], \forall n\ge 1,$$
for some $C_a>0$. On the other hand, the upper bound $4$ of $|Q_b'|$ can be replaced by some $1<R_b<2$, which only needs to satisfy that
 $$  |\brac{Q_b^n}'(y)|\le C_bR_b^n\,,\quad \forall y\in \hI_b, \forall n\ge 0,$$
 for some $C_b>0$. See \cite[Lemma 3.1]{BST} for the existence of $R_b$. Then to guarantee that our argument works, $m_0$ only needs to satisfy that
 $$\mu_a^{m_0}>R_b,$$
 but several details will be more redundant. For simplification, let us work with the assumption $ \lambda_a^{m_0}> 4$.
\end{rmk}

\subsection{Expanding coordinates}\label{subsec:EC}

Define

$$u=u_a:[a-a^2,a]\to I_a,\quad x\mapsto \int_0^x\rho_a(t)\dif t\,,$$
where
$$I_a:=u([a-a^2,a]).$$
Let $h_0$ be the conjugated map of $Q_a:[a-a^2,a]\circlearrowleft$ via $u$, i.e.
$$h_0:I_a\to I_a, \quad h_0=u\circ Q_a\circ u^{-1}.$$
Note that $u(0)=0$ is the unique turning point of $h_0$.  By definition,
\begin{equation}\label{eqn:h_0'}
  |h_0'(u(x))|=\frac{\rho_a(Q_ax)}{\rho_a(x)}|Q_a'(x)|\ge\lambda_a, \quad\forall x\in [a-a^2,a].
\end{equation}

For every $n\ge 0$, let

$$\mathcal{Q}_n:=\big\{\omega\subset I_a: \omega \mbox{ is a connected component of }I_a\setminus h_0^{-n}(u(PC_a))\big\}.$$
Then $\Brac{\mathcal{Q}_n}_{n\ge 0}$ is a sequence of nested Markov partitions of $h_0$. To piecewise analytically continue $h_0^{-1}$ on each $\omega\in\mclq_0$, let us consider the following composition of maps

$$\hat{u}:\pi^{-1}([a-a^2,a])\to I_a, \quad \hat{p}\mapsto u(\pi(\hat{p})).$$

Given a pair $(\omega_0,\omega_1)\in\mclq_0\times\mclq_1$ with $h_0(\omega_1)=\omega_0$, denote $\tau=(h_0|_{\omega_1})^{-1}$, $J_i=u^{-1}(\omega_i)$, $i=0,1$, and let $\sigma:J_0\to J_1$ be the corresponding branch of $Q_a^{-1}$. For $i=0,1$, let $\hat{\omega}_i$  be an arbitrary connected component of $\pi^{-1}(J_i)$, and let $\hat{\tau}:\hat{\omega}_0\to\hat{\omega}_1$ be the lift of $\sigma$ via $\pi$. Then we have the following commutative diagram.

\begin{equation}\label{eqn:diagram}
   \begin{CD}
\hat{\omega}_0 @>\hat{\tau}>> \hat{\omega}_1 \\
@V\pi VV  @VV \pi V          \\
J_0 @>\sigma>> J_1            \\
@V u VV  @VV u V          \\
\omega_0 @>\tau >> \omega_1
\end{CD}
\end{equation}

\begin{lem}\label{lem:continuation}
For the notations introduced above, the following statements hold.

\begin{enumerate}
  \item [(i)] $\hat{u}$ can be analytically continued to a univalent holomorphic function on some open neighborhood of $\overline{\hat{\omega}_0}$ in $\mbbd$.
  \item [(ii)] $\tau$ can be analytically continued to a holomorphic function on some open neighborhood of $\overline{\omega_0}$ in $\mbbc$. Moreover, $c\in \overline{\omega_0}$ is a critical point of $\tau$ precisely if $c\in \partial\omega_0$, $c\ne u(a)$ and $\tau (c)\notin u(PC_a)$. Besides, once $\tau'(c)=0$, then $\tau''(c)\ne 0$.
\end{enumerate}
\end{lem}

\begin{proof}
Since $u'=\rho_a$, by (\ref{eqn:rho_a}), $u'>0$ on $J_0$, and $u$ is real analytic on $J_0$ with singular points of order $\frac{1}{2}$ at both end points. Due to (i) of Lemma \ref{lem:cover}, $\pi'\ne 0$ on $\hat{\omega}_0$ and $\pi$ has critical points of local degree $2$ at both end points of $\hat{\omega}_0$. Therefore, it is easy to see that assertion (i) holds.

Since $\tau=\hat{u}\circ\hat{\tau}\circ\hat{u}^{-1}$ on $\omega_0$, assertion (ii) follows immediately from (iii) of Lemma~\ref{lem:cover} and assertion (i).

\end{proof}

Given a bounded interval $I\subset \mathbb{R}$ and $\xi>0$, let

$$B_\xi(I)=\{z\in\mathbb{C}: \textrm{dist}(z, I)<\xi\}.$$
The following two lemmas are the main results of this subsection, which motivate us to replace $g$ with $h=u\circ g\circ u^{-1}$ as our base dynamics for taking advantage of nice analytical properties of $h^{-1}$.
\begin{lem}\label{lem:expandingcoor}
There exist $\xi>0$ and $D_{\xi,i}>0$, $\forall i\ge 1$, such that the following statements hold.
\begin{enumerate}
\item[(i)] Given $\omega\in \mathcal{Q}_0$, $u^{-1}|_\omega$ can be analytically continued to a holomorphic function on $B_{\xi}(\omega)$.
\item[(ii)] Given $(\omega_0,\omega_1)\in\mclq_0\times\mclq_1$ with $h_0(\omega_1)=\omega_0$, $(h_0|_{\omega_1})^{-1}$ can be analytically continued to a holomorphic function $\tau$ on $B_\xi(\omega_0)$ and $|\tau'|\le \tilde{\lambda\,}_{\!a}^{-1}$ on $B_\xi(\omega_0)$. Moreover, for each $0<\xi'\le\xi$ and each interval $I\subset\omega_0$, $\tau(B_{\xi'}(I))\subset B_{\xi'}(\tau(I))$.
\item[(iii)]Given $n\ge 1$ and $(\omega_0,\omega_n)\in\mclq_0\times\mclq_n$ with $h_0^n(\omega_n)=\omega_0$\,, $(h_0^n|_{\omega_n})^{-1}$ can be analytically continued to a holomorphic function $\tau_n:B_{\xi}(\omega_0)\to B_{\xi}(\omega_n)$. Moreover,

\begin{equation}\label{eqn:tau^(i)}
    \sup_{z\in B_{\xi}(\omega_0)}|\tau_n^{(i)}(z)|\le D_{\xi,i}\tilde{\lambda\,}_{\!a}^{-n},\quad\forall i\ge 1.
\end{equation}

\end{enumerate}
\end{lem}

\begin{proof}
Let $\hat{\omega}$ be an arbitrary connected component of $\hat{u}^{-1}(\omega)$. Then $u^{-1}=\pi\circ (\hat{u}|_{\hat{\omega}})^{-1}$ on $\omega$, and hence (i) of Lemma \ref{lem:continuation} implies that $u^{-1}$ can be analytically continued to a holomorphic function on some open neighborhood of $\overline{\omega}$ in $\mbbc$. Since $\mclq_0$ is a finite set, assertion (i) holds for some $\xi=\xi_1>0$.

On the one hand, according to (ii) of Lemma \ref{lem:continuation}, $(h_0|_{\omega_1})^{-1}$ can be extended to a holomorphic function $\tau$ on some open neighborhood of $\overline{\omega_0}$. On the other hand, (\ref{eqn:h_0'}) implies that $|\tau'|\le\lambda_a^{-1}<\tilde{\lambda\,}_{\!a}^{-1}$ on $\omega_0$.  Noting that $\mclq_0\times \mclq_1$ is a finite set, there exists $\xi_2>0$, independent of $(\omega_0,\omega_1)$, such that $\tau$ is holomorphic on $B_{\xi_2}(\omega_0)$ and $|\tau'|\le \tilde{\lambda\,}_{\!a}^{-1}$ on $B_{\xi_2}(\omega_0)$. For each $0<\xi'\le\xi_2$ and each interval $I\subset\omega_0$, $B_{\xi'}(I)$ is a convex subset of $B_{\xi_2}(\omega_0)$. It follows that $|\tau'|\le \tilde{\lambda\,}_{\!a}^{-1}<1$ on $B_{\xi'}(I)$, and consequently $\tau(B_{\xi'}(I))\subset B_{\xi'}(\tau(I))$. This completes the proof of (ii) for $\xi=\xi_2$.

Consider $(h_0^n|_{\omega_n})^{-1}$ as a composition of $n$ copies of $h_0^{-1}$ and apply (ii) repeatedly for $\xi=\xi_2$. Therefore, it can be extended to a holomorphic function $\tau_n:B_{\xi_2}(\omega_0)\to B_{\xi_2}(\omega_n)$ and $|\tau_n'|\le \tilde{\lambda\,}_{\!a}^{-n}$ on $B_{\xi_2}(\omega_0)$. Take $\xi=\min(\xi_1,\frac{\xi_2}{2})$, so that both (i) and (ii) hold simultaneously. (\ref{eqn:tau^(i)}) follows from Cauchy's estimate.
\end{proof}

\begin{lem}\label{lem:h-distortion}
 There exists $C_d>0$, such that for any $\omega_n\in \mathcal{Q}_n$($n\ge 1$) and any measurable set $E\subset \omega_n$\,, we have
\begin{equation}\label{eqn:h-distortion}
 C_d^{-1}\cdot |h_0^n(E)|^2 \le \frac{|E|}{|\omega_n|} \le C_d \cdot |h_0^n(E)|.
\end{equation}
\end{lem}

\begin{proof}
Let $N=N(a)\ge 1$ be the minimal integer such that for each $\omega\in\mclq_N$, $\partial\omega$ contains at most one point in $h_0^{-1}(u(PC_a))$. For each $n\ge 0$, denote $\mclq_{n+N}$ by $\tilde{\mclq}_n$. It suffices to prove the lemma for $\{\tilde{\mclq}_n\}_{n\ge0}$ instead of $\{\mclq_n\}_{n\ge0}$, because of the reasons below:
\begin{itemize}
  \item  there exists an integer $M=M(a)$, such that for each $n\ge 0$, every element in $\mclq_n$ is a union of no more than $M$ elements in  $\tilde{\mclq}_n$ with a finite set;
  \item once the left inequality in (\ref{eqn:h-distortion}) has been proved for $\{\tilde{\mclq}_n\}_{n\ge0}$, it follows that there exists $p>0$, such that for every $n\ge 0$ and every pair $(\omega,\tilde{\omega})\in\mclq_n\times\tilde{\mclq}_n$ with $\tilde{\omega}\subset\omega$, we have $|\tilde{\omega}|\ge p|\omega|$.
\end{itemize}


To begin with, recall the statement in (ii) of Lemma \ref{lem:expandingcoor}\,, which says that, for each pair $(J_0,J_1)\in\tilde{\mclq}_0\times\tilde{\mclq}_1$ with $h_0(J_1)=J_0$, $(h_0|_{J_1})^{-1}$ can be extended to a holomorphic function $\zeta:B_\xi(J_0)\to B_\xi(J_1)$. Moreover, for each $0<\xi'\le\xi$, and each subinterval $I$ of $J_0$, $\zeta(B_{\xi'}(I))\subset B_{\xi'}(\zeta(I))$. By the choice of $\tilde{\mclq}_0$, $\partial J_0\cap u(PC_a)$ contains at most one point. According to (ii) of Lemma \ref{lem:continuation}, we can fix $0<\xi'\le\xi$ small, independent of $(J_0,J_1)$, such that:

\begin{itemize}
  \item if $\partial J_0\cap u(PC_a)=\{c\}$, $c\ne u(a)$ and $\zeta(c)\notin PC_a$, then $c$ is the unique critical point of $\zeta$ in $B_{\xi'}(J_0)$ and $\zeta''(c)\ne 0$;
  \item otherwise, $\zeta:B_{\xi'}(J_0)\to B_{\xi'}(J_1)$ is univalent.
\end{itemize}

Given $\omega_n\in\tilde{\mclq}_n$, for each $0\le i\le n-1$, let $\omega_i=h_0^{n-i}(\omega_n)$,  let $J_i$ be the unique element in $\tilde{\mclq}_0$ containing $\omega_i$, and let $\zeta_i:B_{\xi'}(J_{i-1})\to B_{\xi'}(J_i)$ be the analytic continuation of $(h_0|_{\omega_i})^{-1}$. Then $(h_0^n|_{\omega_n})^{-1}$ can be extended to a holomorphic function $\tau_n:B_{\xi'}(\omega_0)\to B_{\xi'}(\omega_n)$, where $\tau_n=\zeta_n\circ\cdots\circ\zeta_1$.  To proceed, a classified discussion of the possibility of $\{\partial\omega_i\cap u(PC_a)\}_{i=0}^n$ is needed. There are four cases.

\begin{enumerate}
  \item [1.]$\partial\omega_0\cap u(PC_a)=\varnothing$. Denote $\sigma_n=u^{-1}\circ\tau_n\circ u=(Q_a^n|_{u^{-1}(\omega_n)})^{-1}$. Let $J$ be the element in $\mclq_0$ containing $\omega_0$. By definition, $\sigma_n$ can be extended to a univalent holomorphic function on $\mbbc_{u^{-1}(J)}:=(\mbbc\setminus\mbbr)\cup u^{-1}(J)$. For each $\delta>0$, denote $B_\delta(\omega_0)\cap\mbbr$ by $I_\delta$. Since $(u^{-1})'\ne 0$ on $J$ and since $\tilde{\mclq}_0$ is a finite set, there exists $\delta_0>0$, independent of $\omega_0$, such that $u^{-1}$ can be extended to a univalent holomorphic function on $B_{2\delta_0}(\omega_0)$, which satisfies that $I_{2\delta_0}\subset J$ and $u^{-1}(B_{2\delta_0}(\omega_0))\subset \mbbc_{u^{-1}(J)}$. Then $\sigma_n\circ u^{-1}$ is univalent on $B_{2\delta_0}(\omega_0)$, so by Koebe distortion theorem, the distortion of $\sigma_n\circ u^{-1}$ on $B_{\delta_0}(\omega_0)$ is only dependent on $\delta_0$. It follows that there exists $C>0$, determined by $\delta_0$ only, such than for either component of $\sigma_n\circ u^{-1}(I_{\delta_0}\setminus \omega_0)$, its length is no less than $C\cdot|u^{-1}(\omega_n)|$, where $u^{-1}(\omega_n)=\sigma_n\circ u^{-1}(\omega_0)$. Also note that $\sigma_n\circ u^{-1}(I_{\delta_0})$ does not intersect $PC_a$, the singular set of $u$. Because all the singularity points of $u$ are of square root type, we can conclude that there exists $\epsilon>0$, only dependent on $\delta_0$, such that $u$ can be extended to a univalent function on $B_\epsilon(u^{-1}(\omega_n))$. Since $\sigma_n\circ u^{-1}$ has bounded distortion on $B_{\delta_0}(\omega_0)$, finally we can find $0<\delta\le\delta_0$, determined by $\epsilon$ only, such that $\tau_n=u\circ\sigma_n\circ u^{-1}$ is univalent on $B_\delta(\omega_0)$. Then the conclusion follows from Koebe distortion theorem.

 \item [2.]$\partial\omega_n\cap u(PC_a)\ne\varnothing$. Then by the choice of $\xi'$, $\zeta_i:B_{\xi'}(J_{i-1})\to B_{\xi'}(J_i)$ is univalent, $\forall 1\le i\le n$. Therefore, $\tau_n:B_{\xi'}(\omega_0)\to B_{\xi'}(\omega_n)$ is univalent, and hence the conclusion also follows from Koebe distortion theorem.

 \item [3.]There exists $1\le i\le n$, such that $0\in\partial\omega_i$. Then $\partial\omega_j\cap u(PC_a)$ consists of a unique point $Q_a^{i-j}(0)$, $j=0,1,\dots, i$, and according to the definition of $\tilde{\mclq}_0$, $J_i\cap u(PC_a)=\varnothing$. Therefore, $\zeta_i\circ\cdots\circ\zeta_1:B_{\xi'}(\omega_0)\to B_{\xi'}(\omega_i)$ is in the situation of Case 2, i.e. it is univalent; $\zeta_n\circ\cdots\circ\zeta_{i+1}:B_{\xi'}(J_i)\to B_{\xi'}(J_n)$ is in the situation of Case 1, so it is univalent on $B_\delta(J_i)$. As a result, for $\delta'=\min(\xi',\delta)$, $\tau_n$ is univalent on $B_{\delta'}(\omega_0)$. Once again, the conclusion is implied by Koebe distortion theorem.

  \item [4.] There exists $1\le i\le n$, such that $\partial\omega_{i-1}\cap u(PC_a)=\{c\}$, $c\ne u(a)$ and $\partial\omega_i\cap u(PC_a)=\varnothing$. Note that $\zeta_i(c)\in\partial\omega_i\cap h_0^{-1}(u(PC_a))\subset \partial J_i$, by the definition of $\tilde{\mclq}_0$, which implies that $\partial J_i\cap u(PC_a)=\varnothing$. On the one hand, by the choice of $\xi'$, $\zeta_i:B_{\xi'}(J_{i-1})\to B_{\xi'}(J_i)$ has a unique critical point $c$ with $\zeta_i''(c)\ne 0$. It implies that there is a constant $C>1$, only dependent on $\xi'$, such that for every measurable set $E_i\subset \omega_i$, we have

      $$ C^{-1}\frac{|h_0(E_i)|^2}{|\omega_{i-1}|^2}\le\frac{|E_i|}{|\omega_i|}\le C\frac{|h_0(E_i)|}{|\omega_{i-1}|}. $$
      On the other hand, similar to the discussion in Case 3, $\zeta_{i-1}\circ\cdots\circ\zeta_1:B_{\xi'}(\omega_0)\to B_{\xi'}(\omega_{i-1})$ is in the situation of Case 2, and $\zeta_n\circ\cdots\circ\zeta_{i+1}:B_{\xi'}(J_i)\to B_{\xi'}(J_n)$ is in the situation of Case 1. Then by Koebe distortion theorem, there exists a constant $\tilde{C}>1$, only dependent on $\delta'$, such that for any measurable sets $E_{i-1}\subset \omega_{i-1}$ and $E_n\subset\omega_n$,

      $$\tilde{C}^{-1}|h_0^{i-1}(E_{i-1})|\le\frac{|E_{i-1}|}{|\omega_{i-1}|} \le \tilde{C}|h_0^{i-1}(E_{i-1})| \ ,\quad \tilde{C}^{-1}\frac{|h_0^{n-i}(E_n)|}{|\omega_i|}\le \frac{|E_n|}{|\omega_n|} \le \tilde{C} \frac{|h_0^{n-i}(E_n)|}{|\omega_i|}.$$
      Combining the estimates above, (\ref{eqn:h-distortion}) follows in this case.
\end{enumerate}
\end{proof}

\subsection{Backward shrinking of $Q_a$}
For further usage, let us give some detailed estimates on the accumulation rate of points in $\mbbc\setminus\mclj_a$ to $\mclj_a$ under iterated action of $Q_a^{-1}$.

\begin{lem}\label{lem:contracting}
For each $m\in\mbbn$, there exists $K_m>0$, such that the following statement holds. Let $z_0\in\mbbc\setminus\mclj_a$ and $z_i:=Q_a^{mi}(z_0)$, $\forall i\ge 0$. Assume that $z_n\notin V_a$ for some $n\in\mbbn$ and let $\sigma$ be the branch of $Q_a^{-mn}$ with $\sigma(z_n)=z_0$. Then there are two cases:
\begin{itemize}
  \item [(i)] If  $z_1\notin Q_a^{-1}(V_a)$, then
  \begin{equation}\label{eqn:outside V_a}
   |z_0|\le K_m \abs{z_n}^{2^{-mn}} \quad\mbox{and}\quad |\sigma'(z_n)|\le K_m 2^{-mn}\abs{z_n}^{2^{-mn}-1}.
  \end{equation}

  \item [(ii)] Otherwise, there exists $0<i<n$, such that $z_j\in Q_a^{-1}(V_a)$ for $0\le j\le i$ and  $z_j\notin Q_a^{-1}(V_a)$ for $i<j\le n$, then

  \begin{equation}\label{eqn:inside V_a}
   |\sigma'(z_n)|\le K_m 2^{-m(n-i)}\lambda_a^{-mi}\abs{z_n}^{2^{-m(n-i)}-1}.
  \end{equation}

\end{itemize}
\end{lem}

\begin{proof}
Denote $F_m:=\overline{\mbbc}\setminus Q_a^{-m-1}(V_a)$. Then $F_m$ is a forward invariant compact subset of the Fatou set of $Q_a$. Note that on the Fatou set of $Q_a$, $\sigma$ can be conjugated to some branch of $z\mapsto z^{2^{-mn}}$ via the B\"{o}ttcher coordinate about $\infty$. In case (i), $z_i\in F_m$, $\forall 0\le i\le n$, and hence the conclusion follows from the compactness of $F_m$.

For case (ii), decompose $\sigma$ into $\sigma_1\circ\sigma_2$, where $\sigma_1$ is the branch of $Q_a^{-mi}$ with $\sigma_1(z_i)=z_0$ and $\sigma_2$ is the branch of $Q_a^{-m(n-i)}$ with $\sigma_2(z_n)=z_i$. Note that $z_j\in F_m$ when $j\ge i$. Therefore, as in case (i), $|\sigma_2'(z_n)|$ has an upper bound of order $2^{-m(n-i)}\abs{z_n}^{2^{-m(n-i)}-1}$. To estimate $|\sigma_1'(z_i)|$, because $z_i=Q_a^{mi}(z_0)\in Q_a^{-1}(V_a)$, we can apply (\ref{eqn:lambda_a}) repeatedly to the orbit $\{Q_a^j(z_0)\}_{j=0}^{mi-1}$. As a result,
$$\rho_a(z_0)|\sigma_1'(z_i)|\le \lambda_a^{-mi}\rho_a(z_i).$$
Since $z_i\in F_m$, $\rho_a(z_i)$ is bounded from above, so the conclusion in case (ii) follows.
\end{proof}

\subsection{A family of functions}

From now on, let us fix an arbitrary integer $m_1\ge m_0$ and choose $g=Q_a^{m_1}$ as our base dynamics. As mentioned in the introduction, we shall study $F$ defined in (\ref{eqn:F}) rather than $\mscf$. Let us summarize some frequently used notations and their basic properties here.
\begin{itemize}
  \item $\lambda_g:=\lambda_a^{m_1}\in(4,2^{m_1}]$ and $h:=h_0^{m_1}=u\circ g\circ u^{-1}$. By definition, $|h'|\ge \lambda_g$ on $I_a$.
  \item $\phi:=\varphi\circ u^{-1}$ on $I_a$. Without loss of generality, assume that $|\phi|\le 1$ on $I_a$.
  \item $\hI_b:=[-\sqrt{2b}, \sqrt{2b}]$. $Q_b$ maps $I_b$ to its interior.
  \item Assume that $\alpha$ is sufficiently small, so that $F$ maps $I_a\times \hI_b$ into itself.
  \item $\mclp_n:=\mclq_{m_1n}$, $\forall n\ge 0$. i.e. $\{\mclp_n\}_{n\ge 0}$ is a nested sequence of Markov partitions of $h$.
  \item Denote $F^n(\theta,y)$ by $(h^n(\theta),f_n(\theta,y))$. Due to (i) of Lemma \ref{lem:expandingcoor} and the definition of $F$, for every $\omega\in\mclp_0$ and every $n\ge 1$, $f_n$ is real analytic on $\overline{\omega}\times\mbbr$.
\end{itemize}

Given a non-constant polynomial $\varphi$, let us introduce a family of functions as below, which is inspired by considering the $\alpha$-linear approximation of the derivative of high $F$-iteration of a horizontal curve. The importance of this family will become clear later. See Lemma \ref{lem:F-invariance} and Lemma \ref{lem:separate}.
\begin{defn}\label{def:mathcalT}
Given $\omega\in\mathcal{P}_0$, an analytic function $T:\omega\to \mathbb{R}$ is said to be in the family $\mathcal{T}_\omega$, if there exist a sequence $\Brac{\omega_n\in \mclp_{n}}$ with $\omega_0=\omega$ and $h\brac{\omega_{n+1}}=\omega_n$, $\forall n\ge 0$, and a sequence $\Brac{c_n\in\mbbr}$ with $\abs{c_n}\le 4^{n-1}$, $\forall n\ge 2$, such that

\begin{equation}\label{eqn:functionT}
T(\theta)=(\phi\circ\tau_1)'(\theta)+\sum^{\infty}_{n=2}c_n (\phi\circ\tau_n)'(\theta)\,, \ \forall \theta\in\omega,
\end{equation}
where $\tau_n=\brac{h^n|_{\omega_n}}^{-1}$.
\end{defn}

Since in (\ref{eqn:functionT}), $|c_n|\le 4^{n-1}$ and $|\tau_n'(\theta)|\le \lambda_g^{-n}$, the series always converges.  Due to (\ref{eqn:tau^(i)}), actually each $T\in\mathcal{T}_\omega$ extends to a holomorphic function defined on $B_{\xi}(\omega)$.

All the useful properties of $\mclt$  are listed in the lemma below. The following notations are used in its statement.

\begin{itemize}
  \item $\omega_c$ denotes the unique element in $\mclp_0$ containing $0$.
  \item $\omega_c^\pm\subset\omega_c$ denote the only two elements in $\mclp_1$ with $0\in\partial\omega_c^\pm$.
  \item $\tau_c^\pm:=(h|_{\omega_c^\pm})^{-1}$ are both defined on $h(\omega_c^+)=h(\omega_c^-)\in\mclp_0$.
\end{itemize}
By definition, $\tau_c^\pm$ satisfy that

\begin{equation}\label{eqn:pm id}
\gamma^\pm:=u^{-1}\circ \tau_c^\pm\circ h\circ u: u^{-1}(\omega_c^\mp)\to u^{-1}(\omega_c^\pm)\,, \quad x\mapsto -x.
\end{equation}

\begin{lem}\label{lem:mathcalT}
Let $\varphi$ be a non-constant polynomial. For each $\omega\in\mathcal{P}_0$, the family $\mclt_\omega$ defined for $\varphi$, considered as a space of holomorphic functions defined on $B_{\xi}(\omega)$, is compact with respect to the compact-open topology. Moreover,
\begin{itemize}
  \item [(i)] There exist $A_n>0$, $n=0,1,2,\dots$, such that for each $T\in\mclt_{\omega}$,
  \begin{equation}\label{eqn:A_n}
|T^{(n)}(\theta)|\le A_n,\quad\forall\theta\in\omega, \forall n\ge 0.
\end{equation}
  \item [(ii)] $0\notin\mclt_\omega$. More specifically, there exist $l_0\in \mbbn$ and $B>0$, such that for each $T\in\mclt_{\omega}$,
  \begin{equation}\label{eqn:B}
\sum_{i=0}^{l_0-1}|T^{(i)}(\theta)|\ge 2B, \quad\forall\theta\in\omega.
\end{equation}
  \item [(iii)] If, additionally, $\varphi$ is of odd degree, then for each $T\in\mclt_{\omega_c}$ and each $D\in[-4,4]$, the two functions
      \begin{equation}\label{eqn:nonodd}
       T^\pm:=(\phi'+ D\cdot T)\circ\tau_c^\pm\cdot(\tau_c^\pm)'\in\mclt_{h(\omega_c^\pm)}
      \end{equation}
      are not identical to each other.
\end{itemize}
\end{lem}

\begin{proof} By (\ref{eqn:tau^(i)}) and (\ref{eqn:functionT}), functions in the family $\mclt_{\omega}$ are uniformly bounded on $B_\xi(\omega)$. Then according to Montel's theorem, $\mclt_{\omega}$ is a pre-compact subset of holomorphic functions on $B_\xi(\omega)$ with respect to the compact-open topology. On the other hand, in (\ref{eqn:functionT}), for each $n\ge 1$, there are only finitely many choices of $\tau_n$, so by the definition of $\mclt_{\omega}$, it is also a closed subset. Therefore, $\mclt_{\omega}$ is compact. Assertion (i) follows from the compactness of $\mclt_{\omega}$ immediately.

To prove (ii) and (iii), let us change back to the $x$ coordinate from the $\theta$ coordinate to make use of the complex dynamics of $Q_a$ on the whole Riemann sphere. (ii) and (iii) will be deduced from:
\begin{clm}
Given $J\in u^{-1}(\mclp_0)$, let $\Brac{\sigma_n}_{n\ge 0}$ be a sequence of real analytic maps defined on $J$ with $\sigma_0=\id_J$ and $g\circ\sigma_n=\sigma_{n-1}$, $\forall n\ge 1$, and let $\Brac{c_n}_{n\ge 1}$ be a sequence of numbers with $|c_n|\le 4^n$, $\forall n\ge1$. Then
\begin{equation}\label{eqn:functionS}
 S:=\varphi'+\sum_{n=1}^{\infty}c_n \brac{\varphi\circ\sigma_n}'
\end{equation}
is not identically zero. Moreover, if $\varphi$ is of odd degree and $J=u^{-1}(\omega_c)\ni0$, then $S$ is not an odd function on $J\cap(-J)$.
\end{clm}

\begin{proof}[Proof of Claim]
Note that $S$ can be analytically continued to a holomorphic function defined on $\mbbc_J:=\brac{\mbbc\setminus\mbbr}\cup J$. The basic idea is to show that around $\infty$, the $\varphi'$ term in the series expression of $S$ is dominating. To this end, take an arbitrary $z\in\mbbc_J\setminus V_a$, and let $n_z$ be the minimal integer such that $\sigma_{n_z}(z)\in Q_a^{-1}(V_a)$. The existence of $n_z$ is guaranteed by the fact that the Fatou set of $Q_a$ equals to the attracting basin of infinity, and the backward $Q_a$-invariance of $V_a$ implies that $\sigma_n(z)\notin Q_a^{-1}(V_a)$ if and only if $n<n_z$. Then we can apply Lemma \ref{lem:contracting} to $z$ for $m=m_1$. Firstly, when $n\le n_z$, by (\ref{eqn:outside V_a}),

$$|\brac{\varphi\circ\sigma_n}'(z)|\le C \cdot 2^{-m_1n}\abs{z}^{\deg\varphi\cdot2^{-m_1n}-1}.$$
Secondly, when $n>n_z$, by (\ref{eqn:inside V_a}),

$$|\brac{\varphi\circ\sigma_n}'(z)|\le C\cdot 2^{-m_1n_z}\lambda_g^{n_z-n}\abs{z}^{2^{-m_1n_z}-1}.$$
Here $C>0$ is independent of $z$ or $n$. Since $\abs{c_n}\le 4^{n-1}$ and $2^{m_1}\ge\lambda_g>4$, combining the two inequalities above, we have:

$$\big|\sum_{n=1}^{\infty}c_n\brac{\varphi\circ\sigma_n}'(z)\big|\le C'\cdot \abs{z}^{\deg{\varphi}\cdot2^{-m_1}-1},$$
where $C'>0$ is independent of $z$. On the other hand, as $|z|\to\infty$, $\abs{\varphi'(z)}\asymp\abs{z}^{\deg\varphi-1}$. Therefore,

\begin{equation}\label{eqn:asymptotic}
    \lim_{\substack{z\in\mbbc_J\\|z|\to\infty}}\frac{S(z)}{\varphi'(z)}=1.
\end{equation}

Now we can complete the proof. Since $\varphi$ is non-constant, (\ref{eqn:asymptotic}) implies that $S$ cannot be identically zero, i.e. the first statement of the claim holds. When $\varphi$ is of odd degree, $\abs{\varphi'(z)+\varphi'(-z)}$ is bounded away from zero as $|z|\to\infty$. When $0\in J$, $\mbbc_{J\cap(-J)}$ is a domain containing $0$ and symmetric about $0$. These facts together with (\ref{eqn:asymptotic}) imply that for $z\in \mbbc_{J\cap(-J)}$, $|S(z)+S(-z)|$ is also bounded away from zero as $|z|\to\infty$. The second statement of the claim follows.
\end{proof}

Now let us prove (ii) and (iii) with the help of the claim. To start with, let us clarify the relation between the family $\mclt$ and functions in the form of  (\ref{eqn:functionS}). Given $\omega\in\mclp_0$ and $T=(\phi\circ\tau_1)'+\sum_{n=2}^{\infty}c_n (\phi\circ\tau_n)'\in \mclt_\omega$, denote $J_1=u^{-1}(\tau_1(\omega))$ and let $S=T\circ(h\circ u)\cdot(h\circ u)'$ on $J_1$. By definition, $\phi\circ\tau_n\circ h\circ u=\varphi\circ \sigma_{n-1}$, where $\sigma_0=\id_{J_1}$ and $g\circ\sigma_n=\sigma_{n-1}$ on $J_1$, $\forall n\ge 1$. It implies that $S=\varphi'+\sum_{n=1}^{\infty}c_{n+1}(\varphi\circ\sigma_n)'$ has the form of (\ref{eqn:functionS}) and it is automatically well defined on $J\supset J_1$ with $J\in u^{-1}(\mclp_0)$.

To prove assertion (ii), for each $T\in\mclt_\omega$, let $S=T\circ(h\circ u)\cdot(h\circ u)'$ on $u^{-1}(\tau_1(\omega))$. The claim says that $S$ is not identically $0$, which implies that $0\notin\mclt_{\omega}$. The rest of assertion (ii) follows from this fact and the compactness of $\mclt_{\omega}$ easily by reduction to absurdity.

To prove assertion (iii), by reduction to absurdity, for $T^\pm$ appearing in (\ref{eqn:nonodd}), suppose that $T^+\equiv T^-$ on $h(\omega_c^\pm)$. Define $S:=(\phi'+ D\cdot T)\circ u\cdot u'$ on $J=u^{-1}(\omega_c)$. By definition, $S$ has the form of (\ref{eqn:functionS}), and $S=T^+\circ(h\circ u)\cdot(h\circ u)'$ on $u^{-1}(\omega_c^+)$. By the assumption $T^+\equiv T^-$ on $h(\omega_c^\pm)$, we have:

$$S=T^-\circ(h\circ u)\cdot(h\circ u)'=S\circ\gamma^-\cdot(\gamma^-)'\quad\mbox{on}\quad u^{-1}(\omega_c^+),$$
where $\gamma^-$ is defined in (\ref{eqn:pm id}) and $\gamma^-=-\id$ on $u^{-1}(\omega_c^+)$ . It implies that $S$ is an odd function on $J\cap(-J)$, which contradicts to the claim and completes the proof.
\end{proof}

\begin{rmk}
It should be noted that without assuming that $\varphi$ is of odd degree, the second statement of the claim in the proof of Lemma \ref{lem:mathcalT}, and therefore assertion (iii) in Lemma \ref{lem:mathcalT}, could fail in general. For example, given $c\in[-4,4]$, let  $\varphi(x)=g(x)+cx$ and let $S=\varphi'+\sum_{n=1}^{\infty}(-c)^n(\varphi\circ\sigma_n)'$ in the form of (\ref{eqn:functionS}). Then actually $S=g'$ is an odd function.
\end{rmk}

\subsection{Building expansion}

Here let us summarize some useful results from \cite[Lemma 2.4,\,2.5]{V} and their proofs. It should be noted that these results are only based on the skew-product form of $F$ and the Misiurewicz-Thurston property of $Q_b$, i.e. they are irrelevant to choice of $h$ and $\phi$.
\begin{lem}\label{lem:BE}
There exist constants $\delta_*>0$, $C_*>0$, $1<\sigma<2$ and an integer $N_\alpha$ with $\sigma^{N_\alpha}\le \alpha^{-1}\le 4^{N_\alpha}$, such that when $\alpha$ is small, for an orbit $(\theta_i,y_i)=F^i(\theta,y)$, $(\theta,y)\in I_a\times\hI_b$, $i=0,1,\dots$, the following statements hold.
\begin{itemize}
  \item [(i)]  If $\abs{y}<2\sqrt{\alpha}$, then $\abs{y_k}\ge\delta_*$, $k=1,\dots, N_\alpha$. Moreover, for each $0<\eta\le\frac{1}{3}$, when $\alpha$ is small enough, $\abs{\pd{f_{N_\alpha}}{y}(\theta,y)}\ge \abs{y}\alpha^{-1+\eta}$.
  \item [(ii)] If $\abs{y_i}\ge\sqrt{\alpha}$, $i=0,1\dots,k-1$, then $\abs{\pd{f_k}{y}(\theta,y)}\ge C_*\sqrt{\alpha}\sigma^k$.  If, in addition, $\abs{y_k}\le\delta_*$, then $\abs{\pd{f_k}{y}(\theta,y)}\ge C_*\sigma^k$.
\end{itemize}
\end{lem}
Following \cite{BST}, in the whole paper, we shall fix
$$\eta=\frac{\log\sigma}{8\log 32}.$$
This choice of $\eta$ is only used once for proving the first Claim in Proposition \ref{prop:M_alpha}.

\section{Admissible curves}\label{sec:AC}

In this section, following previous works, we shall introduce the concept of {\em admissible curves}, which are images of horizontal curves under iteration of $F$. Then we shall study analytical properties of admissible curves and show that they are {\em nearly horizontal} but {\em non-flat}.

To begin with, let us give some frequently used notations. Given a curve $X:I\to \mbbr$ defined on some interval $I$, denote its graph by $\hX$, i.e. $\hX=\{(\theta,X(\theta)):\theta\in I\} $. Note that $X$ and $\hX$ are determined by each other, and by abusing terminology, both of them will be called curves. For $h:\hat{I}\to\hat{J}$, diffeomorphic, $I\subset \hat{I}$, $J=h(I)$ and $X:\hat{I}\to \mbbr$, $F(\hX|_I)$ denotes the graph of the curve defined on $J$ by $h(\theta)\mapsto f_1(\theta,X(\theta))$, $\theta\in I$.

Now we can specify the precise meaning of admissible curves in our situation.

\begin{defn}[Admissible Curve] An analytic function $X$ defined on some $\omega\in\mclp_0$ is called an admissible cuvrve, if there exist $y\in\hI_b$, $n\ge 1$ and $\omega_n\in\mclp_n$, such that for the horizontal curve $Y\equiv y$, $\hX=F^n(\hY|_{\omega_n})$.
\end{defn}

\begin{rmk}
By definition, if $X$ is an admissible curve, then $F^n(\hX)$ splits into a union of admissible curves for each $n\in\mbbn$.
\end{rmk}

The main result of this section is:
\begin{prop} \label{prop:admissible curve}
There exist $l_0\in\mathbb{N}$ and $A>B>0$, such that when $\alpha$ is sufficiently small, any admissible curve $X:\omega\to\mbbr$ satisfies that

\begin{equation}\label{eqn:AB}
A\alpha\ge \sum_{i=1}^{l_0+1}|X^{(i)}(\theta)|\ge\sum_{i=1}^{l_0}|X^{(i)}(\theta)|\ge B\alpha, \quad \forall\theta\in \omega.
\end{equation}

\end{prop}

To prove Proposition \ref{prop:admissible curve}, the basic idea is to approximate the derivative of admissible curves by functions in the family $\mclt$, which is guaranteed by:
\begin{lem}\label{lem:F-invariance}
For each $l\in\mbbn$, there exist $C_i>0,i=0,1,\dots,l$, such that the following statement holds when $\alpha$ is small enough. Let $\omega\in\mclp_0$ and let $X:\omega\to \hI_b$ be an admissible curve. Then there exists $T\in\mclt_\omega$, such that
\begin{equation}\label{eqn:C_i alpha^2}
|(X'-\alpha T)^{(i)}(\theta)|\le C_i\alpha^2, \quad\forall \theta\in\omega, \forall 0\le i\le l .
\end{equation}
Moreover, for each $\omega_1\in\mclp_1$ with $\omega_1\subset\omega$ and each $\theta_0\in h(\omega_1)$, if we denote $\hX_1=F(\hX|_{\omega_1})$, $\tau=(h|_{\omega_1})^{-1}$ and

\begin{equation}\label{eqn:T_1}
  T_1=(\phi\circ\tau)'+Q_b'(X(\tau\theta_0))\cdot T\circ\tau\cdot \tau'\quad\mbox{on}\quad h(\omega_1),
\end{equation}
then $T_1\in\mclt_{h(\omega_1)}$ and (\ref{eqn:C_i alpha^2}) still holds when $(X,T,\omega)$ is replaced by $(X_1,T_1,h(\omega_1))$.
\end{lem}

\begin{proof} By definition, $X$ is the $F^n$ image of a constant curve $Y\equiv y$ for some $n\ge 1$ and $y\in\hI_b$. When $n=1$, $X=\alpha\phi\circ\tau_0+Q_b(y)$ for some inverse branch $\tau_0$ of $h^{-1}$ defined on $\omega$. Since $T:=(\phi\circ\tau_0)'\in\mclt_\omega$ and $X'=\alpha T$, (\ref{eqn:C_i alpha^2}) holds automatically in this case. To prove the lemma in full generality, by induction on $n$, it suffices to prove the following statement: for each $l\in\mbbn$, there exist $C_i>0,i=0,1,\dots,l$, such that if $X'-\alpha T$ satisfies (\ref{eqn:C_i alpha^2}), then for $T_1$ defined in (\ref{eqn:T_1}), $T_1\in\mclt_{h(\omega_1)}$ and
\begin{equation}\label{eqn:C_i alpha^2 for X_1}
|(X_1'-\alpha T_1)^{(i)}(\theta)|\le C_i\alpha^2, \quad\forall \theta\in h(\omega_1), \forall 0\le i\le l.
\end{equation}

For the $\theta_0$ appearing in the statement of the lemma, denote $y_0=X(\tau\theta_0)$ and $D=Q_b'(y_0)=-2y_0$. Since $|D|\le 4$, by the definition of $\mclt$, obviously $T_1\in\mclt_{h(\omega_1)}$. Because

$$X_1(\theta)=\alpha\phi(\tau\theta)+Q_b(X(\tau\theta)),\quad\forall \theta\in h(\omega_1),$$
it follows that

\begin{equation}\label{eqn:0-th derivative}
X_1'-\alpha T_1=[D\cdot(X'-\alpha T)-2(X-y_0)\cdot X']\circ\tau\cdot \tau'\quad\mbox{on}\quad h(\omega_1).
\end{equation}

To complete the proof, we only need to show that $X_1'-\alpha T_1$ satisfies (\ref{eqn:C_i alpha^2 for X_1}). To this end, let us show the existence of $C_i$ inductively on $i$.  Firstly, when $i=0$, by (\ref{eqn:0-th derivative}),

$$\|X_1'-\alpha T_1\|\le |D|\cdot\|X'-\alpha T\|\cdot\|\tau'\|+2\|X'\|^2\cdot\|\tau'\|.$$
Here and below $\|\cdot\|$ denotes the maximum modulus norm of functions.
Note that $|D|\le 4$, $\|\tau'\|\le\lambda_g^{-1}$, $\|T\|\le A_0$ (by (\ref{eqn:A_n})) and $\|X'\|\le \|X'-\alpha T\|+\alpha\|T\|$. Therefore,

$$\|X'-\alpha T\|\le C_0\alpha^2\quad\Rightarrow\quad\|X_1'-\alpha T_1\|\le\lambda_g^{-1}[4C_0+2(C_0\alpha+A_0)^2]\alpha^2.$$
It follows that, if we choose $C_0$ large, say $C_0=\frac{3A_0^2}{\lambda_g-4}$, then when $\alpha$ is small, $X_1'-\alpha T_1$ satisfies (\ref{eqn:C_i alpha^2}) for $i=0$.

Now assume that for some $1\le j\le l$, $C_0,C_1,\dots,C_{j-1}$ have been chosen, so that $X_1'-\alpha T_1$ satisfies (\ref{eqn:C_i alpha^2 for X_1}) for $0\le i\le j-1$. Let us determine the choice of $C_j$. Differentiating (\ref{eqn:0-th derivative}) $j$ times, we obtain that

\begin{equation}\label{eqn:i-th derivative}
(X_1'-\alpha T_1)^{(j)}=[D\cdot(X'-\alpha T)^{(j)}-2(X-y_0)\cdot X^{(j+1)}]\circ\tau\cdot (\tau')^{j+1}+P_j+Q_j.
\end{equation}
Here $P_j$ is a linear combination of $(X'-\alpha T)^{(i)}\circ\tau$, $0\le i\le j-1$, and $Q_j$ is a homogeneous quadratic polynomial of  $(X-y_0)^{(j)}\circ\tau$, $0\le i\le j$. For both $P_j$ and $Q_j$, their coefficients are polynomials of $\tau^{(i)}$, $1\le i\le j+1$. By (\ref{eqn:tau^(i)}), $\|\tau^{(i)}\|\le D_{\xi,i}$. By (\ref{eqn:A_n}), $\|T^{(i)}\|\le A_i$. Moreover, $\|(X'-\alpha T)^{(i)}\|\le C_i\alpha^2$, $i=0,\dots,j-1$.
Therefore, on the one hand, there exists $M_j>0$, 
such that

$$\|P_j\|+\|Q_j\|\le M_j\alpha^2.$$
On the other hand,

$$\|X-y_0\|\le\|X'\|\le A_0\alpha+C_0\alpha^2 \quad\mbox{and}\quad\|X^{(j+1)}\|\le \|(X'-\alpha T)^{(j)}\|+\alpha A_j.$$
Substituting the inequalities above into  (\ref{eqn:i-th derivative}), we can conlude that if $\|(X'-\alpha T)^{(j)}\|\le C_j\alpha^2$, then

$$\|(X_1'-\alpha T_1)^{(j)}\|\le \lambda_g^{-j-1}[4C_j+2(A_0+C_0\alpha)(A_j+C_j\alpha)]\alpha^2+M_j\alpha^2.$$
As a result, if we choose $C_j$ large, say $C_j=2M_j+A_0A_j$, then when $\alpha$ is small, (\ref{eqn:C_i alpha^2}) holds for $X_1'-\alpha T_1$ with $i=j$, which completes the induction, and hence the lemma follows.
\end{proof}

Now we can deduce Proposition \ref{prop:admissible curve} from Lemma \ref{lem:mathcalT} and Lemma \ref{lem:F-invariance}.

\begin{proof}[Proof of Proposition~\ref{prop:admissible curve}]
Let $l_0$,  $B$ be as in Lemma \ref{lem:mathcalT}, let $A=2\sum_{i=0}^{l_0}A_i$, where $A_i$'s are given in (\ref{eqn:A_n}), and let $C_0,C_1, \ldots, C_{l_0}$ be determined by Lemma~\ref{lem:F-invariance} for $l=l_0$. Assume that $\alpha$ is so small that  $(l_0+1)C_i\alpha< \min \{A/2,B\}$, $i=0,1,\dots,l_0$ and that Lemma~\ref{lem:F-invariance} holds for $l=l_0$.

By Lemma \ref{lem:F-invariance}, there exists $T\in\mclt_\omega$, such that $X'-\alpha T$ satisfies (\ref{eqn:C_i alpha^2}) for $l=l_0$. Due to Lemma \ref{lem:mathcalT}, $T$ satisfies (\ref{eqn:A_n}) and (\ref{eqn:B}). Then the choice of constants in the previous paragraph guarantees that (\ref{eqn:AB}) holds for $X$.
\end{proof}

As a direct application of the non-flat property of admissible curves, we have the following control of recurrence to $\abs{y}\le\alpha$, which is an analogue of  \cite[Corollary 5.4]{BST}.

\begin{cor}\label{cor:return to alpha}
There exists $\epsilon_*>0$, such that if $\alpha$ is sufficiently small,  then for any admissible curve $X$ defined on $\omega\in\mathcal{P}_0$ and any $0<\epsilon\le\epsilon_*$, we have:
\begin{equation}\label{eqn:admcurnonflat}
\abs{\Brac{\theta\in\omega:\abs{X(\theta)}\le\alpha\epsilon}}\le\epsilon^{\frac{1}{2l_0}}.
\end{equation}

\end{cor}

\begin{proof}
By (\ref{eqn:AB}), for any $\theta_0\in\omega$, there exists $1\le i\le l_0$ with $|X^{(i)}(\theta_0)|\ge B\alpha/l_0$. Since $\|X^{(i+1)}\|\le A\alpha$, it follows that if $|\theta-\theta_0|\le \frac{B}{2Al_0}$, then $|X^{(i)}(\theta)|\ge\frac{B\alpha}{2l_0}$. Therefore, we can divide $\omega$ into a disjoint union of intervals $J_j$, $j=1,2,\dots,m$ with the following properties:

\begin{itemize}
  \item $|J_j|\ge \frac{B}{2Al_0}$;
  \item there exists $1\le i_j\le l_0$, such that $|X^{(i_j)}(\theta)|\ge\frac{B\alpha}{2l_0},\forall \theta\in J_j$.
\end{itemize}
Then $m\le\frac{2Al_0|I_a|}{B}$, and by \cite[Lemma 5.3]{BST}, for any $1\le j\le m$ and any $\epsilon>0$,

$$\big|\big\{\theta\in J_j:\abs{X(\theta)}<\alpha\epsilon\big\}\big|<2^{i_j+1}\brac{\frac{2l_0\epsilon}{B}}^{\frac{1}{i_j}}.$$
Let $M=\max_{1\le j\le m}2^{i_j+1}\brac{\frac{2l_0}{B}}^{\frac{1}{i_j}}$. Then

$$\abs{\Brac{\theta\in\omega:\abs{X(\theta)}\le\alpha\epsilon}}\le mM\epsilon^{\frac{1}{l_0}}\le\frac{2Al_0|I_a|}{B}M\epsilon^{\frac{1}{l_0}},\ \forall \epsilon\in(0,1).$$
Choosing $\epsilon_*\in (0,1)$ with $\frac{2Al_0|I_a|}{B}M\epsilon_*^{\frac{1}{2l_0}}\le 1$, the conclusion follows.
\end{proof}

\section{Critical return}\label{sec:critical return}

The crucial ingredient in proving that $F$ is non-uniformly expanding is to control the approximation of a typical orbit to the critical set $I_a\times\{0\}$. For this purpose, we shall show that $F$ satisfies the so called \emph{slow recurrence conditions} in Proposition \ref{prop:slowrec}.  Following \cite{V,BST}, the key step to deduce these conditions is to prove Proposition \ref{prop:M_alpha}, which is an analogue of \cite[Lemma 2.6]{V} or \cite[Proposition 5.2]{BST}.

\subsection{A technical proposition}
To begin with, let us prove the following lemma as a substitution of \cite[Lemma 2.7]{V} or \cite[Lemma 5.5]{BST}.

\begin{lem}\label{lem:separate}
There exist $M_*\in\mbbn$ and $\epsilon_0 >0$, such that the following statement holds when $\alpha$ is sufficiently small. For each admissible curve $X$ defined on $\omega_0\in\mclp_0$, there exist $\omega^{\pm}\in\mclp_{M_*}$ with $\omega^{\pm}\subset\omega_0$ and $h^{M_*}(\omega^+)=h^{M_*}(\omega^-)$, such that for $\hZ^{\pm}=F^{M_*}(\hX|_{\omega^{\pm}})$,

\begin{equation}\label{eqn:separate}
   \sup_{\theta\in \mathrm{dom}(Z^{\pm})}|Z^+(\theta)-Z^-(\theta)|\ge \epsilon_0\alpha.
\end{equation}
\end{lem}

\begin{proof}
Recall the notations $\omega_c$ and $\omega_c^\pm$ introduced just ahead of Lemma \ref{lem:mathcalT}. Since $h$ is topologically exact on $I_a$, there exists $M_0\in\mbbn$, such that $h^{M_0}(\omega_0)=I_a$ for each $\omega_0\in\mclp_0$. In particular, given $\omega_0\in\mclp_0$, there exists $\omega\in\mclp_{M_0}$, such that $\omega\subset\omega_0$ and $h^{M_0}(\omega)=\omega_c$. Denote $F^{M_0}(\hX|_{\omega})$ by $\hY$ and denote $D=Q_b'(Y(0))$. Since $Y$ is admissible, according to Lemma \ref{lem:F-invariance}, there exists $T\in\mclt_{\omega_c}$,  such that $Y'-\alpha T$ satisfies (\ref{eqn:C_i alpha^2}) for $l=0$.

Let $\hZ^\pm= F(\hY|_{\omega_c^\pm})$. Then both of $Z^\pm$ are admissible curves defined on $h(\omega_c^\pm)$. Therefore, for $T^\pm\in\mclt_{h(\omega_c^\pm)}$ defined in (\ref{eqn:nonodd}) with $T$ and $D$ given in the previous paragraph, by applying Lemma \ref{lem:F-invariance} again, one can see that $(Z^\pm)'-\alpha T^\pm$ also satisfy (\ref{eqn:C_i alpha^2}) for $l=0$. According to (iii) of Lemma \ref{lem:mathcalT} and the compactness of $\mclt_{h(\omega_c^\pm)}$, there exists $\epsilon_1>0$, independent of $T^\pm$, such that

$$\sup_{\theta\in h(\omega_c^\pm)}|T^+(\theta)-T^-(\theta)|\ge 2\epsilon_1.$$
By (\ref{eqn:C_i alpha^2}), $\|(Z^\pm)'-\alpha T^\pm\|\le C_0\alpha^2$. Therefore, when $\alpha$ is sufficiently small, it follows that

$$\sup_{\theta\in h(\omega_c^\pm)}|(Z^+)'(\theta)-(Z^-)'(\theta)|\ge\epsilon_1\alpha.$$
Due to (\ref{eqn:AB}), $|(Z^{\pm})''|\le A\alpha$ on $h(\omega_c^\pm)$. Then it is easy to see that (\ref{eqn:separate}) holds for $M_*=M_0+1$ and suitable choice of $\epsilon_0$.
\end{proof}

The proposition below is the substitution of \cite[Lemma 2.6]{V} or \cite[Proposition 5.2]{BST} with the same idea of proof. We shall follow the proof in \cite{BST}. The main change of ingredients here is to replace \cite[Corollary 5.4]{BST} and \cite[Lemma 5.5]{BST} with Corollary \ref{cor:return to alpha} and Lemma \ref{lem:separate} in this paper respectively.

Denote

$$M_\alpha=\left[\frac{\log\frac{1}{\alpha}}{\log 32}\right] \quad\mbox{and}\quad r_0(\alpha)=\brac{\frac{1}{2}-2\eta}\log\frac{1}{\alpha}.$$
Note that for $N_\alpha$ introduced in Lemma \ref{lem:BE}, we have $M_\alpha\le \frac{2}{5}N_\alpha$.

\begin{prop}\label{prop:M_alpha}
There exists $\beta_0>0$, such that when $\alpha$ is sufficiently small, for each admissible curve $Y$ defined on $\omega_0\in\mclp_0$ and each $r\ge r_0(\alpha)$, we have
\begin{equation}\label{eqn:beta_0}
  \abs{\{\theta\in\omega_0:|f_{M_\alpha}(\theta,Y(\theta))|\le \sqrt{\alpha}e^{-r} \}}\le e^{-\beta_0 r}.
\end{equation}
\end{prop}

\begin{proof}
When $r\ge \brac{\frac{1}{2}+2\eta}\log\frac{1}{\alpha}$, $\sqrt{\alpha}e^{-r}\le\alpha^{1+2\eta}$. Since $F^{M_\alpha}(\hY)$ splits into a union of admissible curves, the conclusion follows from Corollary \ref{cor:return to alpha} and Lemma \ref{lem:h-distortion} immediately by choosing $\beta_0$ appropriately. Otherwise, it suffices to consider the case $r=r_0(\alpha)$ and let us follow the argument in \cite{BST}. Without loss of genarality, we can assume that there exists $z_0=(\theta_0,y_0)$ on $\hY$, such that $|f_{M_\alpha}(z_0)|\le\sqrt{\alpha}$. For each $i\ge 0$, denote $z_i=F^i(z_0)$. Then by (ii) of Lemma \ref{lem:BE},

\begin{equation}\label{eqn:y-derivative lower bound}
 \big|\pd{f_{M_\alpha-i}}{y}(z_i)\big|\ge C_*\sigma^{M(\alpha)-i},\quad 0\le i< M_\alpha.
\end{equation}

Let us summarize some basic estimates of distortion in \cite{BST} with slight modification. Let $y_i=Q_b^i(y_0)$, $i\ge 0$.  Recall the constant $A$ in (\ref{eqn:AB}) and let $L=\max\{1,A|I_a|\}$. Denote $B_i=[y_i-5^iL\alpha,y_i+5^iL\alpha]$, $i\ge 0$. Then
\begin{itemize}
  \item $Q_b(B_i)\subset B_{i+1}$, $\forall i\ge 0$;
  \item $f_i(\hY)\subset B_i$, $\forall i\ge 0$;
  \item $B_i\cap [-\sqrt{\alpha},\sqrt{\alpha}]=\varnothing$, $0\le i < M_\alpha$.
\end{itemize}
It follows that for $0\le i < j< M_\alpha$,
\begin{equation}\label{eqn:distortion_log2}
    \sum_{k=i}^j\sup_{y,y'\in B_k}\abs{\log|Q_b'(y)|-\log|Q_b'(y')|} \le \log 2.
\end{equation}

Following \cite{BST}, let us define some notations and constants. Firstly, let $\overline{\sigma}=\sqrt{\sigma}$ and introduce\begin{equation}\label{eqn:lambda_j}
\lambda_j=\big|\pd{f_{M_\alpha-j}}{y}(z_j)\big|\,\Big/\,\overline{\sigma}^{M_\alpha-j}\ge C_*\overline{\sigma}^{M(\alpha)-j} ,\  0\le j\le M_\alpha.
\end{equation}
Then obviously $\lambda_j\big/\lambda_{j+1}=|Q_b'(f_j(z_0))|\,\big/\,\overline{\sigma}< 4$, $0\le j< M_\alpha-1$. Secondly, recalling the constant $M_*$ appearing in Lemma \ref{lem:separate}, let  $\kappa>4^{M_*}$ be a constant satisfying  that

$$\kappa\brac{4^{-M_*-1}\kappa\epsilon_0-2A|I_a|-8(1-\overline{\sigma}^{-1})^{-1}}\ge 3.$$
Thirdly, let $0=t_1<t_2<\cdots<t_q\le M_\alpha$ be all the times such that

$$t_{i+1}:=\max\Brac{s:t_i<s\le M_\alpha, \lambda_{t_i}\le 2\kappa \lambda_s}.$$
By definition,

\begin{itemize}
  \item $\lambda_{t_q}\le 2\kappa$;
  \item $\lambda_j<\lambda_{t_{i+1}}$ when $t_{i+1}<j<M_\alpha$;
  \item $\frac{\kappa}{2}<\lambda_{t_i}/\lambda_{t_{i+1}}<4^{t_{i+1}-t_i}\Rightarrow t_{i+1}\ge t_i+M_*$.
\end{itemize}
Finally, let
$$k_0(\alpha):=\max\Brac{1\le i\le q: \lambda_{t_i}\ge 2\kappa e^{-r_0(\alpha)}/\sqrt{\alpha}}.$$
Following Claim 1 in the proof of \cite[Proposition 5.2]{BST}, we have:
\begin{clm}
When $\alpha$ is sufficiently small,

\begin{equation}\label{eqn:k_0(alpha)}
    k_0(\alpha) \ge \eta r_0(\alpha)/\log(2\kappa).
\end{equation}
\end{clm}
\begin{proof}[Proof of Claim]
By definition, $\lambda_{t_i}\le 2\kappa\lambda_{t_{i+1}}$, $i=0,1,\dots,q-1$. Because $\lambda_{t_q}\le 2\kappa$, by the choice of $k_0(\alpha)$, $k_0(\alpha)<q$ and hence

$$2\kappa e^{-r_0(\alpha)}/\sqrt{\alpha}\ge \lambda_{t_{k_0(\alpha)+1}}\ge (2\kappa)^{-k_0(\alpha)}\lambda_{t_1}\ge C_*\cdot(2\kappa)^{-k_0(\alpha)}\cdot\overline{\sigma}^{M_\alpha}.$$
Recall that $r_0(\alpha)=(\frac{1}{2}-2\eta)\log(1/\alpha)$, $M_\alpha=[-\frac{\log\alpha}{\log 32}]$,  $\overline{\sigma}=\sqrt{\sigma}$ and $\eta=\frac{\log\sigma}{8\log 32}$. The conclusion follows.
\end{proof}
Denote
$$E:= \{\theta\in\omega_0:|f_{M_\alpha}(\theta,Y(\theta))|\le \sqrt{\alpha}e^{-r_0(\alpha)} \}.$$
Then our aim is to prove that there exists $\beta_0>0$, independent of $Y$, such that when $\alpha$ is sufficiently small, $\abs{E}\le e^{-\beta_0 r_0(\alpha)}$. For each $1\le i\le k_0(\alpha)$, define

$$\Omega_i:=\Brac{\omega\in\mclp_{t_i}:\overline{\omega}\cap E\ne\varnothing}\quad\mbox{and}\quad E_i:=\cup_{\omega\in\Omega_i}\overline{\omega}\,.$$
By definition, $\cap_{i=0}^{k_0(\alpha)}E_i\supset E$.
\begin{clm}
For each $1\le i < k_0(\alpha)$ and each $\omega_i\in\Omega_i$, there exists $\omega\in \mclp_{t_{i+1}}$, such that $\omega\subset\omega_i$ and $\overline{\omega}\cap E=\varnothing$.
\end{clm}
\begin{proof}[Proof of Claim]
Given $\omega_i\in \Omega_i$, $\hY_i:=F^{t_i}(\hY|_{\omega_i})$ is an admissible curve. According to Lemma \ref{lem:separate}, there exist $\omega^{\pm}\in\mclp_{M_*}$ with $\omega^{\pm}\subset h^{t_i}(\omega_i)$ and $h^{M_*}(\omega^+)=h^{M_*}(\omega^-):=\hat{J}$, such that for $\hZ^{\pm}=F^{M_*}(\hY_i|_{\omega^{\pm}})$, (\ref{eqn:separate}) holds.
Since $t_i+M_*\le t_{i+1}$, in particular, there exist $\omega^{\pm}_{i+1}\in\mclp_{t_{i+1}}$, such that $\omega^{\pm}_{i+1}\subset \omega_i$,
$h^{t_i+M_*}(\omega^+_{i+1})=h^{t_i+M_*}(\omega^-_{i+1}):=J\subset \hat{J}$, and

\begin{equation}\label{eqn:separate on J}
\sup_{\theta\in J}|Z^+(\theta)-Z^-(\theta)|\ge \epsilon_0\alpha.
\end{equation}
Denote $t_{i+1}-t_i-M_*$ by $n$ temporarily. By definition, given $\theta\in\dom(Y^\pm_{i+1})$, there exists $\theta_0\in J$, such that $h^n(\theta_0)=\theta$. Due to (\ref{eqn:distortion_log2}),

$$|Y_{i+1}^+(\theta)-Y_{i+1}^-(\theta)|=|f_n(\theta_0,Z^+(\theta_0))-f_n(\theta_0,Z^-(\theta_0))|\ge\frac{1}{2}\big|\pd{f_n}{y}(z_{t_i+M_*})\big|\times|Z^+(\theta_0)-Z^-(\theta_0)|,$$
where

$$\big|\pd{f_n}{y}(z_{t_i+M_*})\big|\ge 4^{-M_*}\big|\pd{f_{t_{i+1}-t_i}}{y}(z_{t_i})\big| =\frac{\lambda_{t_i}}{4^{M_*}\lambda_{t_{i+1}}}\times\overline{\sigma}^{t_{i+1}-t_i}.$$
Then (\ref{eqn:separate on J}) and the two inequalities above imply that

\begin{equation}\label{eqn:t_{i+1}}
    \sup_{\theta\in\dom(Y^\pm_{i+1})}|Y^+_{i+1}(\theta)-Y^-_{i+1}(\theta)|\ge \frac{1}{2}\frac{\lambda_{t_i}}{4^{M_*}\lambda_{t_{i+1}}}\epsilon_0\alpha.
\end{equation}
For $t_{i+1}\le j \le M_\alpha$, define

$$\Delta_j:=\inf_{\theta^{\pm}\in\omega^{\pm}_{i+1}}|f_j(\theta^+,Y(\theta^+))-f_j(\theta^-,Y(\theta^-))|.$$
By (\ref{eqn:t_{i+1}}) and noting that $|(Y^{\pm}_{i+1})'|\le A\alpha$, we have:
\begin{equation}\label{eqn:Delta_{t_{i+1}}}
 \Delta_{t_{i+1}}=\inf_{\theta^\pm\in\dom(Y_{i+1}^\pm)}|Y^+_{i+1}(\theta^+)-Y^-_{i+1}(\theta^-)|\ge \frac{1}{2}\frac{\lambda_{t_i}}{4^{M_*}\lambda_{t_{i+1}}}\epsilon_0\alpha-2A|I_a|\alpha\ge\brac{4^{-M_*-1}\kappa\epsilon_0-2A|I_a|}\alpha.
\end{equation}
As in the proof of \cite[Proposition 5.2]{BST}, denote $D_j=\min_{y\in B_j}|Q_b'(y)|$. Since

$$\Delta_{j+1}\ge D_j\Delta_j-2\alpha,$$
by induction,

$$\Delta_{M_\alpha}\ge\Delta_{t_{i+1}}\prod_{j=t_{i+1}}^{M_\alpha-1}D_j-2\alpha\brac{1+\sum_{j=t_{i+1}+1}^{M_\alpha-1}\bigg(\prod_{l=j}^{M_\alpha-1}D_l\bigg)}.$$
By (\ref{eqn:distortion_log2}) and (\ref{eqn:lambda_j}) ,

$$\lambda_j\overline{\sigma}^{M_\alpha-j}/2\le\prod_{l=j}^{M_\alpha-1}D_l\le 2\lambda_j\overline{\sigma}^{M_\alpha-j}.$$
Noting that $\lambda_j< \lambda_{t_{i+1}}$ when $j> t_{i+1}$, combing the two inequalities above, it is easy to obtain that

$$\Delta_{M_\alpha}\ge \lambda_{t_{i+1}}\overline{\sigma}^{M_\alpha-t_{i+1}}(\Delta_{t_{i+1}}/2-4\alpha(1-\overline{\sigma}^{-1})^{-1}).$$
Substituting (\ref{eqn:Delta_{t_{i+1}}}) into the inequality above and making use of $\lambda_{t_{i+1}}\ge 2\kappa e^{-r_0(\alpha)}/\sqrt{\alpha}$ and the choice of $\kappa$, we have:
$$\Delta_{M_\alpha}\ge 3 e^{-r_0(\alpha)}\sqrt{\alpha}. $$
That is to say, $\overline{\omega^{\pm}_{i+1}}$ cannot both intersect $E$, i.e. the claim holds.
\end{proof}

Now let us return to the proof of the proposition. Let $N\ge 2M_\alpha/k_0(\alpha)$ be a large integer independent of $\alpha$. Then

$$\#\Brac{1\le i< k_0(\alpha):t_{i+1}-t_i\le N}\ge \frac{k_0(\alpha)}{2}.$$
By Lemma \ref{lem:h-distortion}, there exists $p>0$, only depending on $N$, such that for each pair $\omega'\subset\omega$ with $\omega\in\mclp_n$, $\omega'\in\mclp_{n'}$, $n<n'\le n+N$, we have $\abs{\omega'}\ge p\abs{\omega}$. By the second claim, it follows that
when $t_{i+1}\le t_i+N$, $\abs{E_{i+1}}\le (1-p)\abs{E_i}$. Therefore,
$$\abs{E}\le (1-p)^{k_0(\alpha)/2}.$$
Then by (\ref{eqn:k_0(alpha)}), $\abs{E}\le e^{-\beta_0 r_0(\alpha)}$ for some constant $\beta_0>0$, which completes the proof.
\end{proof}

\subsection{Slow recurrence conditions}\label{sebsec:slowrec}

To make the argument slightly simpler, let us adopt the following weakened forms of Corollary \ref{cor:return to alpha} and Proposition \ref{prop:M_alpha}. There exists $0<\beta\le\min\{\frac{1}{2l_0},\beta_0\}$, such that when $\alpha$ is small enough, for each admissible curve $X:\omega\to\mbbr$, we have:
\begin{itemize}
  \item If $\epsilon\le \alpha^2$, then

  \begin{equation}\label{eqn:beta case1}
    |\{\theta\in\omega:|X(\theta)|\le\epsilon\}|\le \epsilon^{\,\beta}.
  \end{equation}

  \item If $\epsilon\le \alpha^{1-2\eta}$, then

  \begin{equation}\label{eqn:beta case2}
    |\{\theta\in\omega:|f_{M_\alpha}(\theta,X(\theta))|\le\epsilon\}|\le \epsilon^{\,\beta}.
  \end{equation}

\end{itemize}
Based on (\ref{eqn:beta case1}) and (\ref{eqn:beta case2}), following the ``Large deviations'' argument in \cite{V}, we can deduce the following version of slow recurrence conditions.

\begin{prop}\label{prop:slowrec}
There exists $c>0$, such that when $\alpha$ is sufficiently small, the following statement holds. For each $\varepsilon>0$, there exists $\tilde{\delta}=\tilde{\delta}(\varepsilon)\in(0, \frac{1}{2})$\,, independent of $\alpha$, such that when $n\in\mbbn$ is sufficiently large,

\begin{equation}\label{eqn:tail}
    \Big|\Big\{(\theta,y)\in I_a\times\hI_b: \sum_{\substack{0\le i< n\\ |f_i(\theta,y)|< \delta}}\log \frac{1} {|f_i(\theta,y)|}> \varepsilon n\Big\}\Big|\le e^{-c\varepsilon \beta \sqrt{n}},
\end{equation}
where $\delta=\tilde{\delta}\,\alpha^{1-2\eta}$. In particular, for Lebesgue a.e. $(\theta, y)\in I_a\times \hI_b$,

\begin{equation} \label{eqn:slowrec}
\limsup_{n\to\infty} \dfrac{1}{n} \sum_{\substack{0\le i< n\\ |f_i(\theta,y)|< \delta}} \log \frac{1} {|f_i(\theta,y)|} \le \varepsilon.
\end{equation}
\end{prop}

\begin{rmk}
It is well known that, starting from the estimate shown in (\ref{eqn:tail}), kinds of statistical properties beyond the Main Theorem can be obtained. See, for example \cite{A1}, for a survey on this topic.
\end{rmk}
To begin with the proof of Proposition \ref{prop:slowrec}, for each $0<\tilde{\delta}<\frac{1}{2}$,  let $\delta=\tilde{\delta}\,\alpha^{1-2\eta}$, and denote

$$\Delta=\Delta(\delta):=\left[\frac{\log\frac{1}{\delta}}{\log\lambda_g}\right].$$
Fix an arbitrary $y_0\in \hI_b\setminus\{0\}$. Let

$$\mclr=\mclr(y_0):=\{\theta\in I_a: h^n(\theta)\notin u(PC_a) \mbox{ and } f_n(\theta,y_0)\ne 0,\,\forall n\ge 0\}.$$
Note that $\mclr$ is of full Lebesgue measure in $I_a$. For each $k\ge 0$ and each $\theta\in I_a$, if $\theta\in\mclr$ and $|f_k(\theta,y_0)|< \delta$, let  $q_k(\theta)$ be the unique positive integer satisfying

$$ \lambda_g^{-q_k(\theta)-1}< \abs{f_k(\theta,y_0)}\le\lambda_g^{-q_k(\theta)}\,;$$
otherwise, let $q_k(\theta)=0$.
By definition, either $q_k(\theta)=0$ or $q_k(\theta)\ge \Delta$.

Fix $\varepsilon>0$. For $K\in\mbbn$, let

$$E_K=E_K(\varepsilon,\tilde{\delta},y_0):=\bigg\{\theta\in \mclr: \sum_{k=1}^K q_k(\theta)\ge \varepsilon K\bigg\}.$$
By Fubini's theorem, Proposition \ref{prop:slowrec} can be reduced to the following lemma.
\begin{lem}\label{E_K}
When $\alpha$ is small enough, the following statement holds. Given $\varepsilon>0$, there exist $0<\tilde{\delta}<\frac{1}{2}$, independent of $\alpha$ or $y_0$, and $K_0\in\mbbn$, independent of $y_0$, such that when $K>K_0$,

$$\abs{E_K}\le \lambda_g^{-\frac{\varepsilon \beta}{10} \sqrt{K}}.$$
\end{lem}

\begin{proof}
Let $Y$ denote the constant curve $Y\equiv y_0$. We shall make use of the admissibility of pieces of $F^n(\hY)$ for appropriate $n$ repeatedly.

Given $K\in\mbbn$, let

$$E^2_K:=\Brac{\theta\in E_K:\exists 1\le k \le K, q_k(\theta)> \sqrt{K}}\quad\mbox{and}\quad E^1_K:=E_K\setminus E_K^2.$$
By definition,

$$E^2_K\subset \bigcup_{k=1}^K \bigcup_{\omega\in\mclp_k}\Brac{\theta\in \mclr\cap\omega:q_k(\theta)> \sqrt{K}}.$$
For each $k\ge 1$ and each $\omega\in\mclp_k$, $X_\omega:=F^k(\hY|_{\omega})$ is an admissible curve, and

$$h^k(\{\theta\in \mclr\cap\omega:q_k(\theta)> \sqrt{K}\})\subset \{\theta\in h^k(\omega):|X_\omega(\theta)|<\lambda_g^{-\sqrt{K}}\}.$$
Therefore, when $K$ is large, say $K>4\brac{\frac{\log\alpha}{\log\lambda_g}}^2$,  by (\ref{eqn:beta case1}) together with (\ref{eqn:h-distortion}),

$$|\{\theta\in\mclr\cap\omega:q_k(\theta)> \sqrt{K}\,\}|\le C_d  \lambda_g^{-\beta\sqrt{K}}|\omega|.$$
It follows that, when $K$ is large,

\begin{equation}\label{eqn:E_K^2}
|E^2_K|\le C_d |I_a|\cdot K\lambda_g^{-\beta\sqrt{K}} \le \lambda_g^{-\frac{\beta}{2}\sqrt{K}}\le\lambda_g^{-\frac{\varepsilon\beta}{2}\sqrt{K}}.
\end{equation}

Given $K\ge 4 M_\alpha^2/\varepsilon^2$, to estimate the size of $E^1_K$, let us denote $L:=2[\sqrt{K}]$, and for every $0\le p\le L-1$, define

$$\mclm_p:=\Brac{M_\alpha\le k\le K:k\equiv p\mod L}$$
and

$$E^1_{K,p}:=\bigg\{\theta\in E^1_K:\sum_{k\in\mclm_p}q_k(\theta)\ge \frac{\varepsilon K}{2L}\bigg\}.$$
By definition, if $\theta\in E^1_K$ with $K\ge 4 M_\alpha^2/\varepsilon^2$, then $q_k(\theta)\le\sqrt{K}$, $\forall 1\le k\le K$, and hence

$$\sum_{p=0}^{L-1}\sum_{k\in\mclm_p}q_k(\theta)=\sum_{k=M_\alpha}^K q_k(\theta)> \varepsilon K-M_\alpha\sqrt{K}\ge \frac{\varepsilon K}{2}.$$
It implies that $E^1_K\subset\bigcup_{p=0}^{L-1}E^1_{K,p}$.

Now let us fix $K$ and $p$ and estimate $|E^1_{K,p}|$. For each $\mbfr=(r_k)_{k\in\mclm_p}$, denote $\norm{\mbfr}=\sum_k r_k$ and let

$$E^1_{K,p}(\mbfr):=\Brac{\theta\in E^1_{K,p}: q_k(\theta)=r_k \mbox{ for all }k \in\mclm_p}.$$

\begin{clm}
Provided that $\alpha$ is small enough, for each $\mbfr\in\Brac{0,1,\dots,[\sqrt{K}]}^{\mclm_p}$, when $K$ is large(independent of $y_0$), we have
\end{clm}

$$\abs{E^1_{K,p}(\mbfr)}\le \lambda_g^{-\frac{2\beta}{3}\norm{\mbfr}}.$$

\begin{proof}[Proof of Claim]

When $\norm{\mbfr}=0$, there is nothing to prove. Otherwise, let $M_\alpha\le k_1<k_2<\dots<k_l$ be all the elements in $\mclm_p$ with $r_{k_j}>0$.
For each $1\le j \le l$ and each $\theta\in E^1_{K,p}(\mbfr)$, let $J_j'(\theta)$ be the element in $\mclp_{k_j-M_\alpha}$ containing $\theta$,  and let $J_j(\theta)$ be the element in $\mclp_{k_j+r_{k_j}}$ containing $\theta$. Finally, let

$$\Omega_j'=\bigcup_{\theta\in E^1_{K,p}(\mbfr)}J'_j(\theta) \quad \mbox{and}\quad \Omega_j=\bigcup_{\theta\in E^1_{K,p}(\mbfr)}J_j(\theta).$$
Since $k_{j+1}-k_j \ge  L \ge r_{k_j} +M_\alpha,$ any element of $\mclp_{k_{j+1}-M_\alpha}$ is either  contained in $\Omega_j$ or disjoint from $\Omega_j$, which implies that $\Omega'_{j+1}\subset \Omega_j$.

Let us show that for each $j=1,2,\ldots, l$ and each component $J_j'$ of $\Omega_j'$,

\begin{equation}\label{eqn:Omegashrink}
|\Omega_{j}\cap J_j'|\le 2^\beta C_d\cdot\lambda_g^{-\beta r_{k_j}}|J_j'|.
\end{equation}
Indeed, for each $\theta'\in \Omega_j\cap J_j'$, there exists $\theta\in E^1_{K,p}(\mbfr)$ such that $\theta'\in J_j(\theta)$. Because $|h'|\ge \lambda_g$ on $I_a$ and $h^{k_j}(J_j(\theta))\in\mclp_{r_{k_j}}$, $|h^{k_j}(\theta)-h^{k_j}(\theta')|\le \lambda_g^{-r_{k_j}}|I_a|$.
Since $\hZ=F^{k_j}(\hY|_{J_j(\theta)})$ is a piece of admissible curve and since $q_{k_j}(\theta)=r_{k_j}$, it follows that
$$|Z(h^{k_j}(\theta'))|\le |Z(h^{k_j}(\theta'))-Z(h^{k_j}(\theta))| + |Z(h^{k_j}(\theta))|\le (A|I_a|\cdot\alpha+1)\lambda_g^{-r_{k_j}}\le 2\lambda_g^{-r_{k_j}}.$$
Denote $\hX:=F^{k_j-M_\alpha}(\hY|_{J'_j})$. The inequality above implies that

$$|f_{M_\alpha}(\theta,X(\theta))|\le 2\lambda_g^{-r_{k_j}},\quad \forall\theta\in h^{k_j-M_\alpha}(\Omega_{j}\cap J_j').$$
Since $r_{k_j}\ge \Delta$, $2\lambda_g^{-r_{k_j}}\le 2\delta\le\alpha^{1-2\eta}$. Then by (\ref{eqn:beta case2}),

$$|h^{k_j-M_\alpha}(\Omega_{j}\cap J_j')|\le 2^\beta\lambda_g^{-\beta r_{k_j}}.$$
Due to (\ref{eqn:h-distortion}),
the estimate (\ref{eqn:Omegashrink}) follows. As a result,

$$|\Omega_j|\le 2^\beta C_d\cdot\lambda_g^{-\beta r_{k_j}}|\Omega_j'|.$$
Therefore, an obvious induction on $j$ implies  that

$$|E^1_{K,p}(\mbfr)|\le \abs{\Omega_l}\le 2^{\beta l}C_d^l\cdot\lambda_g^{- \sum_{j=1}^l \beta r_{k_j}}|\Omega_1'|\le \lambda_g^{-\frac{2\beta}{3}\norm{\mbfr}},$$
provided that $\alpha$ is small, and $\norm{\mbfr}\ge\Delta$ is large accordingly.
\end{proof}

Now let us estimate the number of  possible constrained configurations of $\mbfr$. Given $R\in\mbbn$, define

$$Q(K,p,R):=\#\Brac{\mbfr\in\Brac{0,1,\dots,[\sqrt{K}]}^{\mclm_p}:\norm{\mbfr}=R \mbox{ and } E^1_{K,p}(\mbfr)\ne \emptyset }.$$
\begin{clm}
Given $\varepsilon>0$, there exists $0<\tilde{\delta}<\frac{1}{2}$, such that when $K$ is large(independent of $y_0$),
$$Q(K,p,R)\le \lambda_g^{\frac{\beta}{6}R}, \quad\forall R\ge \frac{\varepsilon K}{2L}.$$
\end{clm}

\begin{proof}[Proof of Claim]

Since each nonzero $r_k$ is no less than $\Delta$,

$$Q(K,p,R)\le \sum_{i=1}^{\min(\#\mclm_p,[R/\Delta])}{\#\mclm_p \choose i }{R-i(\Delta-1)-1\choose i-1},$$
where the first binomial coefficient counts the position of nonzero $r_k$'s, and the second counts the distribution of their values.

By Stirling's formula, $R^{-1}\log{R \choose [R/\Delta]}\to 0$ uniformly in $R$ as $\Delta\to \infty$. Therefore, by the definition of $\Delta$, when $\tilde{\delta}$ is sufficiently small, for $1\le i \le [R/\Delta]$, we have:

$${R-i(\Delta-1)-1\choose i-1}\le {R\choose  [R/\Delta]}\le \lambda_g^{\frac{\beta}{12}R}.$$
Since $\#\mclm_p\approx \frac{\sqrt{K}}{2}$, when $R >\sqrt{\Delta K}$,

$$Q(K,p,R)\le 2^{\#\mclm_p}\lambda_g^{\frac{\beta}{12}R}\le \lambda_g^{\frac{\beta}{6}R}.$$
Otherwise, when $\varepsilon\sqrt{K}\le R \le \sqrt{\Delta K}$,

$$Q(K,p,R)\le \lambda_g^{\frac{\beta}{12}R}\sum_{i=1}^{[R/\Delta]}{\#\mclm_p \choose i }\le [R/\Delta]\lambda_g^{\frac{\beta}{12}R}{\#\mclm_p\choose  [R/\Delta]}.$$
Noting that $\#\mclm_p< R/\varepsilon$, in this case the conclusion also follows from Stirling's formula, provided that $\tilde{\delta}$ is small, and $\Delta/\varepsilon$ is large accordingly.
\end{proof}

By combining the last two claims, we have:

$$|E^1_{K,p}|\le\sum_{\mbfr\in\Brac{0,1,\dots,\sqrt{K}}^{\mclm_p}}|E^1_{K,p}(\mbfr)|\le\sum_{R\ge \frac{\varepsilon K}{2L}}\lambda_g^{-\frac{\beta}{2}R}\le (1-\lambda_g^{-\frac{\beta}{2}})^{-1}\lambda_g^{-\frac{\varepsilon\beta}{8}\sqrt{K}}.$$
It follows that, for $K$ large,

$$|E^1_K|\le \sum_{p=0}^{L-1}|E^1_{K,p}|\le L (1-\lambda_g^{-\frac{\beta}{2}})^{-1} \lambda_g^{-\frac{\varepsilon\beta}{8}\sqrt{K}}\le  \lambda_g^{-\frac{\varepsilon\beta}{9}\sqrt{K}},$$
which, together with (\ref{eqn:E_K^2}), completes the proof of the lemma.
\end{proof}

\section{Proof of the main theorem}
To prove the Main Theorem, it suffices to prove the same statements for $F$ instead of $\mscf$. Recall that, when $\alpha$ is small enough, $F$ maps $I_a\times \hI_b$ into itself. Apparently the interesting dynamics of $F$ is concentrated on the invariant set

$$\Lambda=\bigcap_{n=0}^\infty F^n(I_a\times \hI_b),$$
because for $(\theta,y)$ outside $\Lambda$, $f_n(\theta,y)\to \infty$ as $n\to\infty$.

Let us follow some terminology in \cite{ABV}. To start with, note that the base dynamics $h: I_a\to I_a$ is a $C^2$ local diffeomorphism outside a finite set $\mathcal{C}_\theta$ of singularities. Let

$$\mathcal{S}_{\theta}=\mathcal{C}_\theta\times \hI_b\quad\mbox{and}\quad \mathcal{S}_y= I_a\times \{0\}.$$
Then $F$ is a $C^2$ local diffeomorphism outside $\mathcal{S}=\mathcal{S}_\theta\cup\mathcal{S}_y$, the so called \emph{singular set} in \cite{ABV}, and the conditions  (S1)-(S3) in  \cite{ABV} hold for $\beta=1$. The following subsections are devoted to showing that $F$ is \emph{nonuniformly expanding} in the sense of \cite{ABV}.

\subsection{Positive Lyapunov exponents}

Using the same argument as in \cite{V}, it is easy to deduce vertical positive Lyapunov exponent from Proposition \ref{prop:slowrec}. We only need to consider the vertical Lyapunov exponent at the point $(\theta,y)$ where (\ref{eqn:slowrec}) holds.

Denote the $F$-orbit of $(\theta,y)$ by $\{(\theta_i,y_i)\}_{i\ge 0}$. Given $n\in\mbbn$, let $0\le\nu_1<\nu_2\dots<\nu_s\le n$ be all the times $i$ such that $\abs{f_i(\theta,y)}\le \sqrt{\alpha}$. According to Lemma \ref{lem:BE}, we have:
\begin{itemize}
  \item $\nu_{i+1}-\nu_i\ge N_\alpha$, $1\le i< s$, and in particular $n\ge (s-1)N_\alpha$;
  \item $\abs{\pd{f_{N_\alpha}}{y}(\theta_{\nu_i},y_{\nu_i})}\ge |y_{\nu_i}|\alpha^{-1+\eta}$\,,\,$1\le i< s$;
  \item $\abs{\pd{f_{\nu_{i+1}-\nu_i-N_\alpha}}{y}(\theta_{\nu_i+N_\alpha},y_{\nu_i+N_\alpha})}\ge C_*\sigma^{\nu_{i+1}-\nu_i-N_\alpha}$, $1\le i< s$;
  \item $\abs{\pd{f_{\nu_1}}{y}(\theta_0,y_0)}\ge C_*\sigma^{\nu_1}$\quad and \quad$\abs{\pd{f_{n-\nu_s}}{y}(\theta_{\nu_s},y_{\nu_s})}\ge C_*|y_{\nu_s}|\sqrt{\alpha}\sigma^{n-\nu_s}$.
\end{itemize}
For $\varepsilon>0$, let $0<\tilde{\delta}<\frac{1}{2}$ be determined in Proposition \ref{prop:slowrec}. Then by (\ref{eqn:slowrec}),  when $n$ is sufficiently large, we have:

$$\prod_{i=1}^s|y_{\nu_i}|\ge\delta^s\prod_{\substack{1\le i\le s \\|y_{\nu_i}|< \delta}}|y_{\nu_i}|\ge\tilde{\delta}^s\alpha^{(1-2\eta)s}e^{-2\varepsilon n}.$$
Therefore,

$$\abs{\pd{f_n}{y}(\theta,y)}\ge C_*^{s+1}\tilde{\delta}^s\alpha^{\frac{3}{2}-\eta(s+1)} \sigma^{n-(s-1)N_\alpha}e^{-2\varepsilon n}.$$
Since $\sigma^{N_\alpha}\le\alpha^{-1}$,  $\alpha^{-\eta(s+1)}\sigma^{n-(s-1)N_\alpha}\ge \sigma^{\eta n}$. Choose $\varepsilon=\frac{\eta}{6}\log\sigma$. Note that $\tilde{\delta}$ is determined by $\varepsilon$ and $(s-1)N_\alpha\le n$, when $\alpha$ is small enough, $\big|\pd{f_n}{y}(\theta,y)\big|\ge \sigma^{\frac{\eta}{2}n}$ for $n$ large. Then we have:

\begin{prop}\label{prop:verticalexp}
When $\alpha$ is sufficiently small, for Lebesgue almost every $(\theta,y)\in I_a\times\hI_b$,
\begin{equation}\label{eqn:verticallyaexp}
\liminf_{n\to\infty} \frac{1}{n}\log \left|\frac{\partial f_n}{\partial y}(\theta,y)\right|\ge  \frac{\eta}{2}\log\sigma>0.
\end{equation}
\end{prop}
Therefore, the vertical Lyapunov exponent is positive. Since the base map $h$ is uniformly expanding, we have proved that $F$ has two positive Lyapunov exponents.

\subsection{Existence of a.c.i.p.}
We shall apply the results in \cite[Theorem C]{ABV} to obtain the existence of an a.c.i.p. for ${F}$. By definition,
$$DF(\theta,y)=\left(
\begin{matrix}
h'(\theta) & 0 \\
\alpha \phi'(\theta) & Q_b'(y)
\end{matrix}
\right).$$
Recall that $\abs{h'(\theta)}\ge\lambda_g> 4>|Q_b'(y)|$ on $I_a\times\hI_b$, and hence
$$\|DF(\theta,y)^{-1}\|\le (1+C\alpha) |Q_b'(y)|^{-1},$$
where $C>0$ is a constant independent of $\alpha$.
Therefore,
$$\frac{1}{n} \sum_{i=0}^{n-1}\log \|DF(F^i(\theta,y))^{-1}\|\le C\alpha - \frac{1}{n} \log \left|\frac{\partial f_n}{\partial y} (\theta,y)\right|.$$
By (\ref{eqn:verticallyaexp}), for a.e. $(\theta,y)$ we have
$$\limsup_{n\to\infty}\frac{1}{n}\sum_{i=0}^{n-1}\log \|DF(F^i(\theta,y))^{-1}\|\le -\frac{\eta}{3}\log\sigma<0,$$
provided that $\alpha$ is small enough.
This proves Equation (5) in \cite{ABV}.

Now let us check Equation (6) in \cite{ABV}. It reads as follows: for each $\varepsilon>0$, there exists $\delta>0$ such that for a.e. $(\theta,y)$ we have

\begin{equation*}\label{eqn:slowrecvertical}
\limsup_{n\to\infty}\frac{1}{n}\sum_{\substack{0\le i<n\\ |f_i(\theta,y)|<\delta}}\log |f_i(\theta,y)|^{-1}\le \varepsilon;
\end{equation*}
\begin{equation*}\label{eqn:slowrechorizontal}
\limsup_{n\to\infty} \frac{1}{n}\sum_{\substack{0\le i<n\\ \textrm{dist}(h^i(\theta),\mathcal{S}_\theta)<\delta}} \log |\textrm{dist}(h^i(\theta),\mathcal{S}_\theta)|^{-1}\le \varepsilon.
\end{equation*}
The first inequality is simply (\ref{eqn:slowrec}). To obtain the second one, note that $h$ admits an ergodic a.c.i.p. $\mu$, which is equivalent to the Lebesgue measure on $I_a$. Then by Birkhorff's Ergodic Theorem, for Lebesgue a.e. $\theta\in I_a$, the left hand side of the second inequality is equal to $\int_{\textrm{dist}(\theta,\mathcal{S}_\theta)<\delta}\log|\textrm{dist}(\theta,\mathcal{S}_\theta)|^{-1}\dif\mu(\theta)$. Besides, according to (\ref{eqn:h-distortion}), for every measurable set $E\subset I_a$ and every $n\ge0$, $|h^{-n}(E)|\le C_d|I_a||E|$. As a result, $\frac{\dif\mu}{\dif\Leb}\le C_d|I_a|$ on $I_a$. Since $\mathcal{S}_\theta$ is a finite set, it follows that $\int_{I_a}\log|\textrm{dist}(\theta,\mathcal{S}_\theta)|^{-1} \dif\mu(\theta)<\infty$, and hence the the second inequality holds when $\delta>0$ is small enough.

We have checked that all the conditions of \cite[Theorem C]{ABV} are satisfied provided that $\alpha$ is small enough. Thus $F$ has an absolutely continuous invariant measure.

\subsection{Uniqueness of a.c.i.p.}
As shown in Lemma 6.1 of \cite{AV}, ${\Lambda}=F^n(I_a\times \hI_b)$ when $n\ge 2$. By a similar argument as in \cite[Proposition 6.2]{AV}, it is easy to prove that ${F}$ is topologically exact on $\Lambda$. Moreover, by \cite[Lemma 5.6]{ABV}, up to a set of zero Lebesgue measure, the basin of each a.c.i.p. of ${F}^n$ contains some disk. Therefore the a.c.i.p. of ${F}$ is unique.


\begin{thebibliography}{1234}
\bibitem{A1}J. F. Alves, A survey of recent results on some statistical features of non-uniformly expanding maps. Discrete Contin. Dyn. Syst. 15 (2006), no. 1, 1 - 20.

\bibitem{A2}J. F. Alves, SRB measures for non-hyperbolic systems with multidimensional expansion. Ann. Sci. \'{E}cole Norm. Sup. (4) 33 (2000), no. 1, 1 - 32.

\bibitem{ABV}J. F. Alves, C. Bonatti and M. Viana, SRB measures for partially hyperbolic systems with mostly expanding central direction. Invent. Math. 140 (2000), 351 - 398.

\bibitem{AV}J. F. Alves and M. Viana, Statistical stability for robust classes of maps with non-uniform expansion. Ergodic Theory Dynam. Systems 22 (2002)
no. 1, 1 - 32.

\bibitem{BST}J. Buzzi, O. Sester and M. Tsujii, Weakly expanding skew-product of quadratic maps. Ergodic Theory Dynam. Systems, 23(2003),1401 - 1414.

\bibitem{CG} L. Carleson and T. W. Gamelin, Complex Dynamics, Springer-Verlag, 1993.

\bibitem{M} J. Milnor, Dynamics in One Complex Variable, Third Edition, Princeton University Press, 2006.


\bibitem{N} T. Nowicki, Symmetric $S$-unimodal mappings and positive Liapunov exponents. Ergodic Theory Dynam. Systems 5 (1985), no. 4, 611 - 616.


\bibitem{S1} D. Schnellmann, Non-continuous weakly expanding skew-products of quadratic maps with two positive Lyapunov exponents.  Ergodic Theory Dynam. Systems 28 (2008), no. 1, 245 - 266.

\bibitem{S2} D. Schnellmann, Positive Lyapunov exponents for quadratic skew-products over a Misiurewicz-Thurston map. Nonlinearity 22 (2009), no. 11, 2681 - 2695.

\bibitem{V} M. Viana, Multidimnsional non-hyperbolic attractors. Inst. Hautes  \'{E}tudes Sci. Publ. Math. No. 85 (1997), 63 - 96.
\end{thebibliography}
\end{document}